\documentclass[11pt,a4paper,reqno]{amsart}            

\usepackage[english]{babel}
\usepackage{amssymb,amsmath,amsthm}
\usepackage{mathrsfs}
\usepackage{enumitem}
\usepackage[immediate,debrief]{silence}
\usepackage[pagebackref]{hyperref}
\usepackage[normalem]{ulem}

\usepackage{mathtools}
\usepackage{esint}

\usepackage{color}
\usepackage{comment}
\usepackage{scrtime}

\hypersetup{
	colorlinks,
	linkcolor={blue!90!black},
	citecolor={red!80!black},
	urlcolor={blue!50!black}
}
\usepackage{tikz}
\usepackage{tikz-cd}
\usepackage{graphicx}
\usetikzlibrary{shapes.geometric,arrows,fit,matrix,positioning,trees,snakes}
\tikzset
{
	treenode/.style = {circle, draw=black, align=center, minimum size=1cm},
	subtree/.style  = {isosceles triangle, draw=black, align=center, minimum height=0.5cm, minimum width=1cm, shape border rotate=90, anchor=north}
}

\renewcommand\mathcal{\mathscr}
\theoremstyle{plain}
\newtheorem{theorem}{Theorem}[section]
\newtheorem*{theorem*}{Theorem}
\newtheorem{lemma}[theorem]{Lemma}
\newtheorem*{lemma*}{Lemma}

\newtheorem{proposition}[theorem]{Proposition}

\newtheorem*{conjecture*}{Conjecture}

\newtheorem*{notation*}{Notation}
\newtheorem{thmx}{Theorem}

\theoremstyle{remark}
\newtheorem{remark}[theorem]{Remark}
\newtheorem*{remark*}{Remark}

\theoremstyle{definition}

\newtheorem*{definition*}{Definition}

\theoremstyle{remark}

\newtheorem*{example*}{Example}

\numberwithin{equation}{section}

\newcount\quantno
\everydisplay{\quantno=0}\everycr{\quantno=0}
\newcommand\quant{\advance\quantno by1
	\ifnum\quantno=1\qquad\else\quad\fi\forall }

\newcommand\rest[1]{\kern-.1em
	\lower.5ex\hbox{$\scriptstyle #1$}\kern.05em}

\newcommand\funnyk{k\hbox to 0pt{\hss\phantom{g}}}

\newcommand\whH{\widehat{\phantom{G}}\hbox to 0pt{\hss $H$}}

\newcommand\emspace{\hbox to 6pt{\hss}}

\newcommand\supp{\mathrm{supp}}

\DeclareSymbolFont{EUEX}{U}{euex}{m}{n}

\DeclareSymbolFont{euexlargesymbols}{U}{euex}{m}{n}
\DeclareMathSymbol{\intop}{\mathop}{euexlargesymbols}{"52}
\def\int{\intop\nolimits}

\DeclareSymbolFont{euexsymbols}     {U}{euex}{m}{n}
\DeclareMathSymbol{\smallint}{\mathop}{euexsymbols}{"52}

\begin{document}
	
	\title[Long-Time Heat Asymptotics on Homogeneous Trees]{Long-Time Asymptotics for the Heat Kernel and for Heat Equation Solutions on Homogeneous Trees
	}
	\author{Effie Papageorgiou}

	
	
	\thanks{Acknowledgments. This work is funded by the
Deutsche Forschungsgemeinschaft (DFG, German Research Foundation)–SFB-Geschäftszeichen
–Projektnummer SFB-TRR 358/1 2023 –491392403. 
	We are grateful to the reviewer for the detailed and constructive report. Their comments substantially improved the exposition and, in addressing them, we arrived at a sharper formulation of Theorem A}

	\address[Effie Papageorgiou]{Institut f\"ur Mathematik \\ Universit\"at Paderborn\\
		D-33098 Paderborn \\ Germany}
	\email{papageoeffie@gmail.com}

    \curraddr{Department of Mathematics\\
University of Western Macedonia\\
52100 Kastoria\\
Greece}
	
\begin{abstract}
We study the large-time behavior of the continuous-time heat kernel and of solutions to the heat equation on homogeneous trees. First, we derive sharp asymptotic formulas for the heat kernel as $t\to\infty$. Second, using them, we show that solutions with initial data in weighted $\ell^1$ or suitable radial $\ell^1$ classes asymptotically factorize in $\ell^p$ norms, $p\in[1,\infty]$, as the product of a $p$-mass function and the heat kernel. For comparison, the case of the integers shows that a single constant mass governs the asymptotics of solutions to the heat equation for all $p$, emphasizing the influence that the ``negatively curved'' geometry of the graph has on heat diffusion.
\end{abstract}

\maketitle

\section{Introduction}

The study of the heat equation has given rise to two fundamental tools in modern mathematics: the Fourier transform and the heat kernel. Harmonic analysis, understood as the study of the Fourier transform, is one of the main tools used in this paper, while the heat kernel is its central object of study. In the setting of Riemannian geometry, the heat kernel can be constructed intrinsically (see, e.g., \cite{Gri}), revealing a deep connection between its analytic behavior and the geometry of the underlying manifold.

On the other hand, random walks on graphs form a vast subject with strong connections to probability, analysis, geometry, and algebra, and they have long been studied as discrete analogs of diffusions on manifolds. A central theme is the relationship between the geometry of an infinite graph and the asymptotic behavior of the associated random walk. In close connection to the discrete-time heat kernel are \textit{local limit theorems}; however, these deal with fixed vertices, while many applications require uniform asymptotics in large space--time regimes.  
Continuous-time random walks provide a more regular framework than discrete-time walks and may be defined via subordination by a Poisson process; a standard reference is \cite[Section 5]{Barlow}.

 In this paper, we consider the heat semigroup $(e^{-t\mathcal{L}})_{t> 0}$ associated with the natural Laplacian $\mathcal{L}$ on a homogeneous tree, a basic example of a graph with exponential volume growth. 
The first main result of this paper is the derivation of large-time asymptotics for the continuous-time heat kernel on homogeneous trees, in large space-time regimes. To state this result, let us recall that a homogeneous tree $\mathcal{T}$ of degree $q + 1$, $q\geq 2$, is a connected graph with no loops, in which every vertex is adjacent to $q + 1$ other vertices. We denote by $d$ the natural distance on $\mathcal{T}$. Let $o$ be a fixed reference point on $\mathcal{T}$; write $|x|$ for $d(x,o)$.

Set
\[
\gamma(0)=\frac{2\sqrt q}{q+1}
\qquad\text{and}\qquad
b_2=1-\gamma(0).
\]
Define
\begin{equation}\label{eq:def_eta}
\eta(z)=\frac{1+\sqrt{1+z^2}}{z},
\qquad z>0,
\end{equation}
and extend $\eta$ continuously to the interval
$[0,\infty]$ by setting $\eta(0)=\infty$ and $\eta(\infty)=1$.
We also set
\[
H(z)=\frac{q\eta(z)^2}{q\eta(z)^2-1},
\qquad z\in[0,\infty].
\]
The following heat kernel asymptotics are the first main result of this work.

\begin{thmx}\label{thm:thmA}
Let $h_t$ be the heat kernel on $\mathcal{T}$ and let
$h_t^{\mathbb Z}$ be the heat kernel on $\mathbb Z$. If
$t\to\infty$ and $|x|\to\infty$, then
\[
h_t(x)\sim
\frac{2}{\gamma(0)}\,t^{-1}e^{-b_2t}
H\!\left(\frac{t\gamma(0)}{|x|+1}\right)
(1+|x|)q^{-|x|/2}
h_{t\gamma(0)}^{\mathbb Z}(|x|+1).
\]

\noindent Instead, if $x\in \mathcal{T}$ satisfies $|x|\leq \rho(t)$, where $\rho(t)\to \infty$ and  $\frac{\rho(t)}{\sqrt{t}}\to 0$ as $t\to \infty$, then
		\[
		h_t(x)=\frac{\sqrt{2}}{\sqrt{\pi}}\frac{q(q+1)}{(q-1)^2}\gamma(0)^{-3/2}\, t^{-3/2}\,e^{-b_2t}\, \varphi_0(x)\left(1+O\Big(\frac{\rho(t)}{\sqrt{t}}\Big)\right), 
		\]
        where $\varphi_0(x)=\left(1+\frac{q-1}{q+1}|x|\right)q^{-|x|/2}$.
	\end{thmx}

To the best of our knowledge, so far only global bounds (albeit sharp) have been available on homogeneous trees in the literature, see \cite[Proposition 2.5]{CMS}. In contrast, for the continuous-space counterpart, hyperbolic space, global asymptotics for the heat kernel as both time and space tend to infinity were established in \cite{AJ99}. Most of the time, do not pursue quantitative rates for $h_t^{\mathbb Z}$, and consequently for $h_t$, or for caloric functions, as they are not needed for our purposes.

The next main task of the present work is to discuss the large-time behavior of \textit{solutions} to the heat equation on $\mathcal{T}$; in fact, the previous asymptotics on the heat kernel are instrumental to this end, while bounds seem not to suffice. More precisely, the question is motivated by a classical result on Euclidean space, which we now describe, see e.g. \cite{Vaz0}. Let $u_0\in L^1(\mathbb{R}^n)$, denote by $M=\int_{\mathbb{R}^n}u_0(x)\, dx$ its mass, and by $G_t(x) = (4\pi t)^{-n/2}e^{-\frac{|x|^2}{4t}}$ the Euclidean heat kernel. Then for all $1\leq p\leq \infty$,
\begin{equation}\label{eq:conv_Rn}
		\lim_{t \to \infty} 
		\frac{1}{\|G_t\|_{L^p(\mathbb{R}^n)}}
		\| u_0\ast G_t-M\, G_t\|_{L^p(\mathbb{R}^n)}=0.
		\end{equation}
        Notice that Young's inequality only gives $\|u_0\ast G_t\|_{p}\leq \|u_0\|_{1}\|G_t\|_{p}$, so the mass $M$ times $G_t$ ``cancels enough'' with $u_0\ast G_t$ so that \eqref{eq:conv_Rn} holds.
        
Beyond the Euclidean framework, the geometry of the underlying space plays a decisive role. On manifolds with non-negative Ricci curvature, analogous asymptotic results hold, see \cite{GPZ}. In contrast, negatively curved spaces exhibit markedly different behavior. In real hyperbolic spaces $\mathbb{H}^n$, Vázquez showed in \cite{Vaz} that the classical $L^{1}(\mathbb{H}^n)$-asymptotic convergence in \eqref{eq:conv_Rn} holds for \textit{radial} initial data; instead, he showed that the $L^1(\mathbb{H}^3)$ result may fail if the initial condition is non-radial, and that the $L^{\infty}(\mathbb{H}^3)$ result in \eqref{eq:conv_Rn} may fail even with radial symmetry of the solution. These results were later extended in \cite{APZ} to Riemannian symmetric spaces of non-compact type $G/K$, which include all hyperbolic spaces; in the same setting, a correction to the constant (showing it should vary with $p$), so that \eqref{eq:conv_Rn} holds for $p>1$ -albeit just for symmetric initial data- was obtained in \cite{Naik}.

The appropriate approach toward understanding this phenomenon turns out to be the introduction of suitable \emph{mass corrections}, which turn out to be \textit{functions} instead of constants. For symmetric spaces $G/K$, and for any compactly supported initial datum, it was shown in \cite{P1} that the asymptotic behavior of caloric functions, that is, solutions to the heat equation, depends crucially on the value of $p\in[1,\infty]$: for $1\leq p<2$ and $2\leq p\leq\infty$, distinct mass terms arise naturally. In the bi-$K$-invariant setting, these masses reduce to constants, related to the spherical transform of the initial condition. Analogous results were later obtained in \cite{PT} in the non-Archimedean setting of affine buildings (even exotic ones) for discrete-time isotropic random walks, where mass functions were defined using boundary integrals and Macdonald's ground spherical functions, and where the dichotomy at the critical exponent $p=2$ again plays a fundamental role.

The purpose of the present work is to establish the continuous-time analog of these results on \emph{homogeneous trees}. Homogeneous trees occupy a natural position between symmetric spaces and affine buildings: they constitute the simplest example of the latter, but also a discrete analog to real hyperbolic spaces. Even more, heat propagation on homogeneous trees exhibits striking non-Euclidean features, as already observed in \cite{CMS, MedSe} (see also \cite{T} for random walks on affine buildings), as a result of the ``negatively curved'' geometry.

To state our second result, let us introduce some further notation. The geometric boundary $\Omega$ of a homogeneous tree $\mathcal{T}$ can be viewed as the set of all geodesic rays tipped at $o$. For
	any $x \in  \mathcal{T}$, we set $\Omega(o, x)=\{\omega\in \Omega: \, \text{the geodesic ray } [o,\omega) \, \text{contains } x\}$. The boundary 
	carries the harmonic measure $\nu$ centered at $o$, such that $\nu(\Omega(o, x))$ depends only on the distance
	$d(o, x)$. Recall that we say that a function $f$ is radial if $f(x)$ depends only on $d(o, x)$. 
    Let $h_{\omega}(x)$ be the Busemann (or horocycle or height) function for $x \in  \mathcal{T}$ and $\omega \in \Omega$; see Subsection \ref{sec:HA} for the definition and more details.
	
	We next introduce mass functions. \footnote{We use the term \textit{mass} with some abuse of terminology. In the present paper, the relevant quantities are not necessarily nonnegative, and may even depend on the spatial variable. Thus, the terminology should not be understood as implying either nonnegativity or constancy. A term such as \textit{charge}  would arguably be more precise. We retain the conventional terminology for consistency with the existing literature.} To this end, we need some weighted $\ell^1$ spaces. We say that $f$ belongs to $\ell^1(w_p)$,
	$p \in [1, \infty]$, if
	\[
	\sum_{y \in \mathcal{T}} |f(y)| w_p(y) < \infty,
	\]
	where the weight $w_p$ is defined by
	\begin{equation}\label{eq:defwp}
	w_p(y) =
	\begin{cases} q^{\frac{1}{p}|y|}  &\text{if } p \in [1, 2), \\
		q^{\frac{1}{2}|y|} , &\text{if } p \in [2, \infty].
	\end{cases}
	\end{equation}
	Now, for $f \in \ell^1(w_p)$, the mass function $M_pf$ is defined by the formula
	\begin{align*} 
		M_pf(x) &= \sum_{y \in \mathcal{T}} f(y)	
		\fint_{\Omega(o,x)} q^{\frac{1}{p}h_{\omega}(y)}\, d\nu(\omega), \qquad &\text{if } p \in [1, 2), \\
		\intertext{or}
		M_pf(x) &= \frac{1}{\varphi_0(x)} f\ast\varphi_0(x), &\text{if } 
		p \in [2, \infty);  
	\end{align*}
in fact, it is a bounded function on $\mathcal{T}$. The mass function $M_pf$ is also well-defined for \textit{radial} functions $f$ belonging to $ \ell^1_{\delta_p}(\mathcal{T})$; this space is defined by
\begin{equation}\label{eq:defdeltap}		
\ell^1_{\delta_p}(\mathcal{T}) =
		\Big\{f : \mathcal{T} \to \mathbb{C}: \sum_{y \in \mathcal{T}} |f(y)| \varphi_{i\delta_p}(y) < \infty
		\Big\},
	\end{equation}
	where $\varphi_{\lambda}$, $\lambda\in \mathbb{C}$, are the radial eigenfunctions of the Laplace operator $\mathcal{L}$ satisfying the normalization condition $\varphi_{\lambda}(o)=1$, and 
    \begin{equation}\label{eq:deltap}
    \delta_p=\begin{cases}
		\frac{1}{p}-\frac{1}{2}, &\quad 1\leq p\leq 2, \\
		0, &\quad 2\leq p\leq \infty.
	\end{cases}
    \end{equation}
In this case, namely when the initial condition is in $ \ell^1_{\delta_p}(\mathcal{T})$ and radial, mass functions reduce to \emph{constants}:
	\begin{equation*}
		M_pf \equiv \sum_{y\in \mathcal{T}} f(y) \varphi_{i\delta_p}(y).
	\end{equation*}
The second main result is as follows.

	\begin{thmx}\label{thm:thmB}
		Let $h_t$ be the continuous time heat kernel on a homogeneous tree, associated with the Laplacian $\mathcal{L}$. 
		If $f \in \ell^1(w_p)$, $p \in [1, \infty]$, or $f$ is radial and 
		belongs to $\ell^1_{\delta_p}(\mathcal{T})$, then
		\[ 
		\lim_{t \to \infty} 
		\frac{1}{\|h_t\|_{\ell^p}}
		\| e^{-t\mathcal{L}}f-M_pf\cdot h_t\|_{\ell^p}=0.
		\]
	\end{thmx}
	The main class of functions in $\ell^1(w_p)$ is those that are finitely supported. Indeed, the result for the other classes follows by density arguments; notice that
	\[
	\text{finitely supported}\subseteq \ell^{1}(w_p)\subseteq \ell^{1},
	\]
	but 
	\[
	\text{finitely supported radial }\subseteq \ell^{1}_{\text{rad}}\subseteq \ell^{1}_{\delta_p}\text{ radial}.
	\]
	It is worth mentioning that for radial functions in $\ell^1(\mathcal{T})$
	we have 
	\[
	M_1(f) \equiv \sum_{y \in \mathcal{T}} f(y),
	\]
	which is reminiscent of \eqref{eq:conv_Rn}, and it is the discrete-space, continuous-time analog of results on hyperbolic space or Riemannian symmetric spaces of the non-compact type (continuous time-space) or random walks on affine buildings (discrete time-space). The crucial ingredient in all these non-positively curved settings seems to be a careful analysis of the heat kernel, see e.g. \cite{AJ99, APZ, P1, PT, T}, especially in the so-called \textit{heat concentration regions}, where the interplay of the geometry versus the fast decay in space of the heat kernel becomes apparent. It is now that Theorem \ref{thm:thmA} becomes essential. Remarkably, we obtain here the same mass functions as for (isotropic, aperiodic, irreducible, finite-range) random walks on homogeneous trees in \cite{PT}: the transition densities for discrete time, as well as their concentration regions, seem to memorize information for the walk, but the mass functions do not, see e.g. \cite{PT, T}.

	For comparison, we recall that in the case of $\mathbb{Z}$, the constant $M=M(f)=\sum_{\mathbb{Z}}f$ works for all $p\in[1,\infty]$ and all $f\in \ell^1$, as proven in \cite{Abadiasetal}. This complies with the result \eqref{eq:conv_Rn} on $\mathbb{R}^n$; see also \cite{MedSe} for a comparison between the heat kernel on $\mathbb{Z}$ and the heat kernel on a homogeneous tree with $q\geq2$, in terms of heat concentration.

This paper is organized as follows: In Section \ref{sec:1}, we present some preliminaries on homogeneous trees and the
corresponding harmonic analysis. In Section \ref{sec:2} we discuss the heat kernel on homogeneous trees, its connection with the heat kernel on $\mathbb{Z}$,  as
well as $\ell^p$ heat concentration on critical regions. Section \ref{sec:3} is devoted to obtaining asymptotics for the heat kernel, hence containing the proof of Theorem \ref{thm:thmA}. It also contains asymptotics for ratios of heat kernels in certain regions, as an indispensable technical tool for the proof of Theorem \ref{thm:thmB}, which deals with the asymptotic behavior of caloric functions, that is, solutions to the heat equation. Finally, Section \ref{sec:T} deals with the
asymptotic behavior of caloric functions on homogeneous trees of exponential volume growth. We introduce mass functions for $p\in [1, \infty]$, and we discuss
their properties. We next show that for finitely supported initial data, the corresponding solutions to the heat equation, converge in the $\ell^p$
critical region to the mass function times the heat kernel as time grows to infinity. On the other hand, we
show that both the solution and the heat kernel vanish asymptotically outside the critical regions. In the last
subsection, we discuss these problems for more general initial data in the $\ell^p$, $p\in [1, \infty]$ setting, thus completing the proof of Theorem \ref{thm:thmB}.

 Throughout this paper, we use the convention that $c, C...$
stand for a generic positive
constant whose value can change from line to line. We denote $\{0,1,2,...\}$ by $\mathbb{N}_0$. The notation $A\lesssim B$ between two positive expressions
means that there is a constant $c>0$ such that $A\leq c \, B$. The notation $A\asymp B$ means that $A\lesssim B$ and $B\lesssim A$. Finally, we write $x_j\sim y_j$ if $\lim_{j}\frac{x_j}{y_j}=1$, or equivalently $x_j=y_j(1+o(1))$.

	\section{Homogeneous trees}\label{sec:1}
	A homogeneous tree\index{homogeneous tree} $\mathcal{T}$ of degree $q+1$ is an infinite connected graph with no loops, in which every vertex is adjacent to $q+1$ other vertices. When $q=1$, a homogeneous tree can be identified with the set of integers. From now on, we assume that $q\geq 2$ unless specified, and we identify $\mathcal{T}$ with its set of vertices. For more details, see \cite{CMS0, CMS, Figa}.
	
	A homogeneous tree carries a natural distance $d$ and a natural measure $\mu$. Specifically, $d(x,y)$ is the number of edges of the shortest path joining $x$ to $y$ and $\mu$ is the counting measure. For the counting measure, the volume of any sphere $S(x,n)$ in $\mathcal{T}$ is given by
	\[
	|S(x,n)|=\begin{cases}
		1, & \text{if } n=0\\
		(q+1)q^{n-1}, & \text{if } n\in \mathbb{N},
	\end{cases}
	\]
	while for $n \geq 1$, the volume of any ball $B(x,n)$ is comparable to $q^{n}$.
	
	Let us fix a base point $o$ and set $|x|=d(x,o)$. A function
$f:\mathcal T\to\mathbb C$ is called radial if its value depends only
on $|x|$. We identify every radial function with its radial profile on
$\mathbb{N}_0$ and use the same symbol for both. Thus, if $f$ is radial,
then $f(n)$ denotes its value at any vertex of distance $n$ from $o$.
In particular, we may write
\[
f(x)=f(|x|),\qquad
f(d(x,y)),
\]
where in the second expression the radial profile of $f$ is evaluated
at $d(x,y)$. We use this convention throughout.

	We may define the convolution of two functions $f_1, f_2$ on $\mathcal{T}$,  provided $f_2$ is radial, by
	\begin{equation}\label{convradial}
		(f_1\ast f_2)(x)=\sum\limits_{n\geq 0}f_2(n)\sum\limits_{d(x,y)=n}f_1(y).
	\end{equation}
	Finally, we denote by  $\ell^p(\mathcal{T})$ the Lebesgue spaces  associated to the measure $\mu$, with norms given by
	\[
	\|f\|_{\ell^p(\mathcal{T})}=
	\begin{cases}
		\left(\sum\limits_{x\in\mathcal{T}}|f(x)|^p \right)^{1/p}, & \text{if } 1\leq p<\infty\\
		\sup\limits_{x\in\mathcal{T}}|f(x)|, & \text{if } p=\infty.
	\end{cases}
	\]
	
	\subsection{Boundary, Spherical functions and Harmonic analysis} \label{sec:HA}

	A useful reference for this section is \cite{CMS0}. 
	
	A geodesic ray in $\mathcal T$ is a sequence of vertices
$(x_0,x_1,x_2,\ldots)$ such that
\[
d(x_j,x_k)=|j-k|,
\qquad j,k\in\mathbb N_0.
\]
Two geodesic rays $(x_n)_{n\geq0}$ and $(x_n')_{n\geq0}$ are said to
be equivalent if they eventually coincide up to a shift of their
parametrizations; that is, if there exist $m,n\in\mathbb N_0$ such
that
\[
x_{n+j}=x'_{m+j},
\qquad j\in\mathbb N_0.
\]
This defines an equivalence relation on the set of geodesic rays.
The boundary $\Omega$ of $\mathcal T$ is the set of the resulting
equivalence classes. Every equivalence class contains a unique
geodesic ray starting at $o$.

For $\omega\in\Omega$, write $[o,\omega)=(v_0,v_1,v_2,\ldots)$, $v_0=o$,
for the unique geodesic ray starting at $o$ and representing $\omega$.
The height function associated with $\omega$ is defined by
\[
h_\omega(x)
=
\lim_{m\to\infty}\bigl(m-d(x,v_m)\bigr),
\qquad x\in\mathcal T.
\]
In particular, $h_\omega(o)=0$. In some references, the Busemann
function is defined with the opposite sign; we follow the convention
of \cite{CMS0}. Since $d(o,v_m)=m$, the triangle inequality gives
\[
m-|x|\leq d(x,v_m)\leq m+|x|.
\]
Consequently,
\begin{equation}\label{eq:Busemann_bds}
-|x|\leq h_\omega(x)\leq |x|,
\qquad \omega\in\Omega.
\end{equation}

There is a unique Borel probability measure $\nu$ on $\Omega$ such that, for every $n\geq1$ and every
$x\in\mathcal T$ with $d(o,x)=n$,
\[
\nu(\Omega(o,x))
=
{q^{1-n}}{(q+1)^{-1}},
\]
where, for each $n\geq1$, the sets $\{\Omega(o,x):d(o,x)=n\}$
form a partition of $\Omega$ into open and compact $(q+1)q^{n-1}$ sets, \cite[p.5]{Figa}.

	\subsection{The spherical Fourier transform on homogeneous trees}\label{sec:FT}
	Our main aim in this section is to define the spherical Fourier transform and its inverse on a homogeneous tree $\mathcal{T}$. For more details, see \cite{CMS0, CMS, Figa, VECA}.
	
	Let $\mathcal{M}$ be the mean operator
	\begin{equation*}
		(\mathcal{M}f)(x)=\frac{1}{q+1}\sum\limits_{y\in \mathcal{T},\; d(x,y)=1}f(y).
	\end{equation*} Then, the Laplacian $\mathcal{L}$ on $\mathcal{T}$ is defined by
	\[\
	\mathcal{L}=I-\mathcal{M};
	\]
	it is easily seen to be bounded on $\ell^p(\mathcal{T})$ for every $p$ in $[1, \infty]$, and self-adjoint on $\ell^2(\mathcal{T})$. The real segment $[1-\gamma(0),1+\gamma(0)]$, where $\gamma(0)=2\sqrt{q}/(q+1)$, coincides with the $\ell^2$ spectrum of $\mathcal{L}$.

    The \textit{spherical function} $\varphi_\lambda$, $\lambda\in\mathbb C$, is
the unique radial eigenfunction of $\mathcal M$ satisfying
$\varphi_\lambda(o)=1$. Its eigenvalue is
	\begin{equation}\label{gammalambda0}
		\gamma(\lambda)=\frac{q^{i\lambda}+q^{-i\lambda}}{q^{\frac{1}{2}}+q^{-\frac{1}{2}}}=\gamma(0)\cos(\lambda \log q).
	\end{equation}
    Set 
	$$\tau=\frac{2\pi}{\log q}.$$ 
	 The spherical functions have the following explicit expressions, see e.g. \cite[p.1273]{CMS}:
	\begin{equation} \label{eq:sphericalT}
		\varphi_\lambda(x)=
		\begin{cases}
			\left( 1+\frac{q-1}{q+1}|x| \right)q^{-|x|/2} & \text{if } \lambda \in \tau\mathbb{Z},\\
			\left( 1+\frac{q-1}{q+1}|x| \right)q^{-|x|/2}(-1)^{|x|} & \text{if } \lambda \in \frac{\tau}{2}+\tau\mathbb{Z},\\
			\textbf{c}(\lambda)q^{( -\frac{1}{2}+i\lambda)|x|}+\textbf{c}(-\lambda)q^{( -\frac{1}{2}-i\lambda)|x|} & \text{if } \lambda \notin (\frac{\tau}{2})\mathbb{Z}, \\
		\end{cases}
	\end{equation}
	where $\textbf{c}$ is the meromorphic function
	\begin{equation} \label{cfunc}
		\textbf{c}(z)=\frac{1}{q^{1/2}+q^{-1/2}} \frac{q^{1/2+iz}-q^{-1/2-iz}}{q^{iz}-q^{-iz}}, \quad \forall z \in \mathbb{C} \backslash (\frac{\tau}{2})\mathbb{Z}.
	\end{equation}
	Note that $\varphi_\lambda$ is periodic with period $\tau$, $\varphi_{\lambda}=\varphi_{-\lambda}$ and that $\varphi_{i/2}\equiv 1$. An integral representation of these elementary spherical functions, see \cite[p.400]{CMS0} or \cite[Theorem 2.1]{Figa}, is given by 
\begin{equation}\label{eq:integralrepspherical}
		\varphi_{\lambda}(x)=\int_{\Omega} {q^{(\frac{1}{2}+i\lambda)h_\omega(x)} }\,d\nu(\omega), \quad x\in \mathcal{T}.
	\end{equation}
	More generally, viewed from the reference point $x$, one has for the ground spherical function
\begin{equation}\label{eq:integralrepspherical0}
\varphi_0(d(x,y))=
		\int_\Omega q^{\frac{1}{2}h_{\omega}(x)}\,
		 q^{\frac{1}{2}h_{\omega}(y)}   d\nu( \omega).
\end{equation} 
This symmetric formula can be found in \cite[p.55]{Figa} (with a different parametrization: our $\varphi_0$ corresponds to the parameter $s=1/2$ of \cite{Figa}).

	The \textit{spherical Fourier transform}\index{spherical Fourier transform} of a radial function $f\in \ell^1(\mathcal{T})$ is defined by
	\[\
	\mathcal{H}f(\lambda)=\sum\limits_{x\in \mathcal{T}}f(x)\varphi_\lambda(x)=f(0)+\sum\limits_{n\geq 1}(1+q)q^{n-1}f(n)\varphi_\lambda(n), \quad  \lambda \in \mathbb{C}.
	\]
	Since $\varphi_\lambda=\varphi_{-\lambda}$ and $\varphi_{\lambda+\tau}=\varphi_{\lambda}$, $\mathcal{H}f$ is even and $\tau$-periodic. Then, the following inversion formula holds true:
	\begin{equation*} 
		f(x)=\frac{q\log q}{4\pi(q+1)}\int_{-\tau /2}^{\tau /2}\mathcal{H}f(\lambda)\varphi_{\lambda}(x)\frac{d \lambda}{|\textbf{c}(\lambda)|^2}, \quad  x\in \mathcal{T},
	\end{equation*}		
	where the Plancherel density for $\lambda \in [-\tau/2, \tau/2)$ is equal to
	\begin{equation}\label{eq:Planchereldens}
		|\mathbf{c}(\lambda)|^{-2}=\frac{4(q+1)^2 \sin ^2(\lambda \log q)}{(q+1)^2 \sin ^2(\lambda \log q)+(q-1)^2 \cos ^2(\lambda \log q)}. 
	\end{equation}

	\subsection{The Helgason-Fourier transform}
	The Helgason-Fourier transform $\widehat{f}$ of a finitely supported function $f$ on $\mathcal{T}$ is a function on $\mathbb{C} \times \Omega$ 
	defined by the formula
	\begin{equation}\label{eq:HFtransform}
	\widehat{f}(\lambda,\omega)=\sum_{x\in \mathcal{T}}f(x)\,q^{(\frac{1}{2}+i\lambda)h_{\omega}(x)}.
\end{equation}
	It is clear that $\widehat{f}(\lambda,\omega)=\widehat{f}(\lambda+\tau,\omega)$ for every $\lambda\in \mathbb{C}$. If $f$ is radial, then its Helgason-Fourier transform becomes independent of the variable $\omega$
	and
	$$\widehat{f}(\lambda,\omega)=\mathcal{H}f(\lambda)=\sum_{x\in \mathcal{T}}f(x)\varphi_{\lambda}(x),$$
	that is, the Helgason-Fourier transform reduces to the spherical Fourier transform (see also \cite[p.401]{CMS0}). 
	
	\section{The heat kernel on homogeneous trees}\label{sec:2}
	We now discuss the notion of the continuous-time heat kernel on homogeneous trees. For more details, see \cite{CMS, MedSe}.

Since $\mathcal L$ is bounded on $\ell^p(\mathcal T)$ for every
$p\in[1,\infty]$, the exponential $e^{-t\mathcal L}$ defined by
\[
e^{-t\mathcal L}
=
\sum_{n=0}^\infty\frac{(-t\mathcal L)^n}{n!},
\qquad t\geq0,
\]
converges in the uniform operator topology.
The heat semigroup generated by the Laplacian $\mathcal{L}$ is denoted by $(e^{-t\mathcal{L}})_{t> 0}$ and it is uniformly
continuous on $\ell^p(\mathcal{T})$. At the kernel level, the heat kernel is $h_t(x,y)=e^{-t\mathcal{L}}\delta_y(x)$. 
The self-adjointness of $e^{-t\mathcal{L}}$ yields that the heat kernel is symmetric in $x,y$; the identity $e^{-t\mathcal{L}}1=1$ then gives
\[
1=\sum_{y\in \mathcal{T}}h_t(x,y)=\sum_{x\in \mathcal{T}}h_t(x,y)= \sum_{x\in \mathcal{T}}h_t(y,x) \quad \forall t\geq 0, \, y\in \mathcal{T},
\]
showing that heat is conserved.

	The heat kernel $h_t$ is given as the inverse spherical Fourier transform of the function 
	\begin{equation*}
	m_t(\lambda)=e^{-t(1-\gamma(\lambda))},
	\end{equation*} where $\gamma(\lambda)$ is defined on (\ref{gammalambda0}). By evenness, it holds 
	\begin{equation}\label{inversionT} 
		h_t(x)=\frac{q\log q}{2\pi(q+1)}\int_{0}^{\tau /2}e^{-t(1-\gamma(\lambda))}\varphi_{\lambda}(x)\frac{d \lambda}{|\textbf{c}(\lambda)|^2}, \quad  x\in \mathcal{T}.
	\end{equation}
Since $\varphi_{\lambda}$ is radial, so is $h_t$. Thus
\[
h_t(x,y)=h_t(d(x,y)),
\]
and in particular, $h_t(x)=h_t(x,o)=h_t(|x|)$.

Moreover, for any $f\in \ell^p(\mathcal{T})$, $1\leq p \leq \infty$, we have
$$e^{-t\mathcal{L}}f(x)=\sum_{y\in \mathcal{T}}h_t(x,y)f(y)=f\ast h_t(x),$$
namely the norm analytic solution in $\ell^p(\mathcal{T})$ of the heat equation $$(\partial_t+\mathcal{L})u=0, \,\, \forall (x,t)\in \mathcal{T}\times (0, \infty), \quad u(x,0)=f(x).$$
	
	Formulas for the heat kernel on $\mathcal{T}$ are closely related to those for the heat kernel on the set of integers $\mathbb{Z}$ (which can be identified with a homogeneous tree with $q=1$). More precisely, the Laplacian on the integers $\mathbb{Z}$ is given by 
	\[
	\mathcal{L}^{\mathbb{Z}}f(j)=f(j)-\frac{f(j+1)+f(j-1)}{2}, 
	\]
	and the corresponding heat kernel $h_t^{\mathbb{Z}}$ is given by
	\begin{align*}
		h_t^{\mathbb{Z}}(j)=\frac{e^{-t}}{2\pi}\int_{-\pi /2}^{\pi/2}e^{t\cos s}\cos (sj)ds=e^{-t}I_{|j|}(t), 
	\end{align*}
	for every $j\in \mathbb{Z}$ and $t>0$, where $I_{|j|}(t)$ is the modified Bessel function of imaginary
argument. It is clear that $h_t^{\mathbb{Z}}$ has sum equal to $1$ too (i.e., it is a probability measure on $\mathbb{Z}$). The following lemma by Cowling, Meda and Setti, describes the behavior of $h_t^{\mathbb{Z}}$.
	
	\begin{theorem}\label{thm:heat Z asym}(\cite[Theorem 2.3]{CMS})
		Let $\xi:\mathbb{R}^{+}\to \mathbb{R}$, $F:\mathbb{Z}\times \mathbb{R}^{+}\to \mathbb{R}$ denote the functions defined by the rules 
		$$\xi(z)=(1+z^2)^{1/2}+\log \frac{z}{1+(1+z^2)^{1/2}}$$
		and 
		$$F(j,t)=\begin{cases}
			(2\pi)^{-1/2}\frac{\exp\{-t+|j|\,\xi(t/|j|)\}}{(1+j^2+t^2)^{1/4}}, & \text{if } j\neq 0\\
			(2\pi)^{-1/2}(1+t^2)^{-1/4}, & \text{if } j=0.
		\end{cases}$$
		Then 
		$$h_t^{\mathbb{Z}}(j)\asymp F(j,t) \text{ for all } j\in \mathbb{Z}, \text{ for all } t>0,$$ and 
		$$h_t^{\mathbb{Z}}(j)\sim F(j,t) \text{ as } j^2+t^2 \text{ tends to } \infty.$$
	\end{theorem}

	We now return to the heat kernel on trees. The result in \cite{CMS} given below, connects $h_t$ with $h_t^{\mathbb{Z}}$.
	\begin{proposition}\label{heatkernel}(\cite[Proposition 2.5]{CMS})
		The following hold for every $t>0$ and $x\in \mathcal{T}$:
		
		(i) $h_t(x)=e^{-b_2t}q^{-|x|/2}\sum_{k=0}^{\infty}q^{-k}\left[h_{t\gamma(0)}^{\mathbb{Z}}(|x|+2k)-h_{t\gamma(0)}^{\mathbb{Z}}(|x|+2k+2)  \right],$
		
		(ii) $h_t(x)=\frac{2e^{-b_2t}}{\gamma(0)t}q^{-|x|/2}\sum_{k=0}^{\infty}q^{-k}\,(|x|+2k+1)\,h_{t\gamma(0)}^{\mathbb{Z}}(|x|+2k+1),$
		
		(iii) $h_t(x)\asymp \frac{e^{-b_2t}}{t}(1+|x|)\,q^{-|x|/2}\,h_{t\gamma(0)}^{\mathbb{Z}}(|x|+1).$
	\end{proposition}
\begin{remark}
    Formula (ii) above will be our main tool for the asymptotics of $h_t$ on $\mathcal{T}$; naturally, we will use the asymptotics of Theorem \ref{thm:heat Z asym} for $h_t^{\mathbb{Z}}$. Since the latter are stated in Theorem \ref{thm:heat Z asym} without rates, we will likewise not keep track of them in the resulting asymptotics on $\mathcal{T}$. We note that quantitative rates may be available from more refined asymptotics on $\mathbb{Z}$, but establishing or exploiting such rates is beyond the scope of the present work.
\end{remark}
    
    The $\ell^p$ norms of the heat kernel can be derived either from pointwise estimates of the heat kernel, \cite[p.1740]{MedSe}, or by using harmonic analysis, \cite[Lemma 2.1]{CMS}:
    \begin{equation}\label{eq:heatLpnorms}
        \|h_t\|_{\ell^p}\asymp \begin{cases}
            \min\{1, t^{-\frac{1}{2p'}}\}\,e^{-b_p t}, \quad &\text{if} \quad 1\leq p <2;\\
            \min\{1, t^{-\frac{3}{4}}\}\,e^{-b_2 t}, \quad &\text{if} \quad p=2;\\
           \min\{1, t^{-\frac{3}{2}}\}\,e^{-b_2 t}, \quad &\text{if} \quad 2 < p \leq \infty.
        \end{cases}
    \end{equation}
	Here, for all $1\leq p\leq 2$, we set
    \[
    b_p=1-\gamma\left( i \Big(\frac{1}{p}-\frac{1}{2}\Big)\right),
    \]
    which actually coincides with the infimum of the real part of the bottom of the $\ell^p$ spectrum of $\mathcal{L}$, \cite[p.4275]{CMS}.
    
	We now introduce a notion that is essential for the $\ell^p$ concentration of heat on homogeneous trees. Let $1\leq p < \infty$. We say that $\textbf{B}_p(t)\subseteq \mathcal{T}$ is an $\ell^p$ \emph{critical region} for the heat kernel on $\mathcal{T}$, if
	\begin{equation}\label{eq:critical}
	\frac{1}{\|h_t\|_{\ell^p}}\Bigg(\sum_{x\in\textbf{B}_p(t)}h_t(x)^p\Bigg)^{1/p}\to 1, \quad \text{as} \quad  t\to \infty, 
	\end{equation}
	or equivalently, if 
	\[
	\frac{1}{\|h_t\|_{\ell^p}}\Bigg(\sum_{x\in\mathcal{T}\setminus \textbf{B}_p(t)}h_t(x)^p\Bigg)^{1/p}\to 0, \quad \text{as} \quad  t\to \infty.
	\]
    
	Using the above-mentioned heat kernel estimates, Medolla and Setti determined in \cite{MedSe} the $\ell^p$ critical regions $\textbf{B}_p(t)$ of $h_t$:
    
	\begin{theorem}(\cite[p.1740]{MedSe}) \label{prop:MedSe_ellp} The following space-time regions $\textbf{B}_p(t)$ describe where the heat kernel on $\mathcal{T}$ concentrates in  $\ell^p$, in the sense of \eqref{eq:critical}:		
	\begin{itemize}
			\item[(i)] Let $1 \leq p<2$. Set $R_p = (q^{1/p} - q^{1/p'})/(q+1)$. If $r(t)$ is any function such that $r(t) / t^{1 / 2} \to\infty$ and $r(t)/t\to 0$ as $t \to\infty$, then
			$$
			\left\|h_t\right\|_p^{-p} \sum_{R_p t-r(t) \leq|x| \leq R_p t+r(t)} h_t(x)^p \to 1 \quad \text { as } t \to\infty .
			$$
			\item[(ii)] Let $p=2$. If $r_1(t) / t^{1 / 2} \to 0$ and $r_2(t) / t^{1 / 2} \to\infty$ as $t \to\infty$, then
			$$
			\left\|h_t\right\|_2^{-2} \sum_{r_1(t) \leq|x| \leq r_2(t)} h_t(x)^2 \to 1 \quad \text { as } t \to\infty .
			$$
			\item[(iii)] Let $2<p<\infty$. If $r_3(t) \to\infty$ as $t \to\infty$, then
			$$
			\left\|h_t\right\|_p^{-p} \sum_{|x| \leq r_3(t)} h_t(x)^p \to 1 \quad \text { as } t \to\infty;
			$$
			if $p=\infty$, then 
			$$
			\left\|h_t\right\|_{\infty}^{-1} \sup _{|x| \geq r_3(t)} h_t(x) \to 0 \quad \text { as } t \to\infty.
			$$
		\end{itemize}
	\end{theorem}
 Woess gave an interesting probabilistic interpretation -and a strengthening- of this phenomenon for $p=1$ in \cite{Woess}, in terms of a central limit theorem: for the heat diffusion process $(X_t)_{t\geq  0}$, $d(X_t, X_0)$ is asymptotically normal with mean $R_1t$ and variance $t$ (see also \cite{Bab} for a probabilistic interpretation for a similar phenomenon on $L^1$ on symmetric spaces). However, there is no probabilistic interpretation known, to the best of our knowledge, for $p>1$, which makes the analytic statement indispensable.
	
	\section{Heat kernel asymptotics}\label{sec:3}
	The purpose of this section is to give results concerning large-time heat kernel asymptotics, that is, to prove Theorem \ref{thm:thmA}. So far, to the best of our knowledge, only (sharp) bounds were available in the literature \cite{CMS}, so these results are new; in addition, to study the large-time asymptotic behavior of solutions to the heat equation, see Theorem \ref{thm:thmB}, bounds would not be sufficient. On hyperbolic space (and in fact on all rank one non-compact symmetric spaces), complete asymptotics for the heat kernel were obtained in \cite[p.1080]{AJ99}, as time and space go to infinity; moreover, these asymptotics were dictated by the ratio of space to time. The main results of this section on homogeneous trees are analogous.
	
For future reference, let us recall the function 
	 $\xi$ defined in Theorem \ref{thm:heat Z asym} and the function $\eta$ defined in \eqref{eq:def_eta}. We note that 
\begin{equation}\label{eq:xieta}
\xi(z)=\sqrt{1+z^2}-\log\eta(z),\quad \xi'(z)=\frac{\sqrt{1+z^2}}{z}, \quad z>0.
\end{equation}

We start with an auxiliary lemma, the heart of the argument, which allows us to reduce the problem to heat kernels on $\mathbb{Z}$. This is essential to determine asymptotics for the heat kernel on $\mathcal{T}$. As already highlighted by Cowling, Meda and Setti in \cite[p.4280]{CMS}, the connection between these two heat kernels is reminiscent of the relationship between the heat kernel of a symmetric space and that of a maximal flat submanifold. 

\begin{lemma}\label{lem:F-ratios-corrected}
Extend $F$ to $(0,\infty)\times(0,\infty)$ by the same formula as in
Theorem \ref{thm:heat Z asym}. Then the following assertions hold.

\begin{enumerate}
\item[\rm(i)]
For every $\alpha\in[0,\infty]$ and every fixed
$m\in \mathbb{Z}$,
\[
\lim_{\substack{\sigma,\tau\to\infty\\ \tau/\sigma\to\alpha}}
\frac{F(\sigma+m,\tau)}{F(\sigma,\tau)}
=
\eta(\alpha)^{-m},
\]
where the right-hand side is understood to be $1$ when $m=0$. Moreover, when $\alpha\neq0$, we have 
\[
\frac{F(\sigma+m,\tau)}{F(\sigma,\tau)}
=
\eta(\alpha)^{-m}(1+o(1)), \quad \sigma, \tau\to \infty, \quad \frac{\tau}{\sigma}\to \alpha.
\]
\item[\rm(ii)]
For every $\alpha\in[0,\infty]$,
\[
\lim_{\substack{\sigma,\tau\to\infty\\ \tau/\sigma\to\alpha}}
\sum_{k=1}^{\infty}
q^{-k}
\left(1+\frac{2k}{\sigma}\right)
\frac{F(\sigma+2k,\tau)}{F(\sigma,\tau)}
=
\sum_{k=1}^{\infty}
\bigl(\sqrt q\,\eta(\alpha)\bigr)^{-2k}.
\]

\item[\rm(iii)]
Without any assumption on the convergence of $\tau/\sigma$, one has
\[
\sum_{k=1}^{\infty}
q^{-k}
\left(1+\frac{2k}{\sigma}\right)
\frac{F(\sigma+2k,\tau)}{F(\sigma,\tau)}
-
\frac{1}{q\eta(\tau/\sigma)^2-1}
\to 0
\]
as $\sigma,\tau\to\infty$.

\end{enumerate}
\end{lemma}

\begin{proof}
We first prove \rm(i). Notice that $\sigma +m>0$ for all $\sigma$ large enough, since $m$ is fixed. We set
\[
\omega_m(\sigma,\tau)
=
(\sigma+m)\xi\left(\frac{\tau}{\sigma+m}\right)
-
\sigma\xi\left(\frac{\tau}{\sigma}\right).
\] 
Then, by the definition of $F$,
\[
\frac{F(\sigma+m,\tau)}{F(\sigma,\tau)}
=
\left(
\frac{1+\sigma^2+\tau^2}
     {1+(\sigma+m)^2+\tau^2}
\right)^{1/4}
e^{\omega_m(\sigma,\tau)},
\]
and we observe that for fixed $m$, the factor with the fourth root tends to $1$ as
$\sigma,\tau\to\infty$. Moreover, for fixed $\tau>0$, the fundamental theorem of calculus for $s\mapsto s\xi\left(\frac{\tau}{s}\right)$ yields
\begin{equation}\label{eq:omega-corrected}
\omega_m(\sigma,\tau)
=\int_\sigma^{\sigma+m}
\frac{d}{ds}\left[ s\xi\left(\frac{\tau}{s}\right)\right]\,ds=
-\int_\sigma^{\sigma+m}
\log\eta\left(\frac{\tau}{s}\right)\,ds
\end{equation}
where for the differentiation we used \eqref{eq:xieta}.

If $\alpha\in(0,\infty)$ and $\tau/\sigma\to\alpha$, then
$\tau/s\to\alpha$ uniformly for
$s$ between $\sigma$ and $\sigma+m$. Hence $\omega_m(\sigma,\tau)
\to 
-m\log\eta(\alpha),$
and therefore
\[
e^{\omega_m(\sigma,\tau)}
\to 
\eta(\alpha)^{-m}.
\]

If $\tau/\sigma\to\infty$, then $\tau/s\to\infty$ uniformly for
$s$ between $\sigma$ and $\sigma+m$. Since $\eta(\infty)=1$, it follows that $
\omega_m(\sigma,\tau)\to 0$ 
and hence
\[
e^{\omega_m(\sigma,\tau)}
\to 1=\eta(\infty)^{-m}.
\]

Finally, if $\tau/\sigma\to0$ and $m> 0$, then $\tau/s\to0$ uniformly for
$s\in[\sigma, \sigma+m]$. Since $\eta(z)\to\infty$ as $z\to0^+$,
\eqref{eq:omega-corrected} gives
$\omega_m(\sigma,\tau)\to -\infty.$
Thus
\[
e^{\omega_m(\sigma,\tau)}
\to 0=\eta(0)^{-m}.
\]
For $m<0$, the interval of integration in $[\sigma+m, \sigma]$, so we get 
\[
e^{\omega_m(\sigma,\tau)}
\to \infty=\eta(0)^{-m}.
\]
For $m=0$, the quotient is identically equal to $1$.

Finally, the stronger assertion when $\alpha\neq 0$ follows directly from the fact that $\eta$ is singular only at $0$, while $\eta(z)\geq 1$ for $z>0$. This proves \rm(i).

We next prove \rm(ii). For fixed $\tau>0$, one has
\[
\frac{\partial}{\partial\sigma}\log F(\sigma,\tau)
=
-\log\eta\left(\frac{\tau}{\sigma}\right)
-\frac{\sigma}{2(1+\sigma^2+\tau^2)}
<0.
\]
Thus $F(\cdot,\tau)$ is decreasing, so for $k\geq 1$, 
\[
0<
\frac{F(\sigma+2k,\tau)}{F(\sigma,\tau)}
\leq1.
\]
Moreover, for $\sigma\geq1$,
\[
1+\frac{2k}{\sigma}\leq1+2k,
\]
and
\[
\sum_{k=1}^{\infty}q^{-k}(1+2k)<\infty.
\]
Part \rm(ii) now follows from part \rm(i) and the dominated convergence
theorem.

Finally, define
\[
S(\sigma,\tau)
=
\sum_{k=1}^{\infty}
q^{-k}
\left(1+\frac{2k}{\sigma}\right)
\frac{F(\sigma+2k,\tau)}{F(\sigma,\tau)}
\]
and
\[
G(r)=\frac{1}{q\eta(r)^2-1},
\qquad r\in(0,\infty),
\]
and $G(0)=0$, 
$G(\infty)=1/(q-1)$. The function $G$ is continuous on the interval
$[0,\infty]$. We claim that
\[
S(\sigma,\tau)-G(\tau/\sigma)\to 0
\]
as $\sigma,\tau\to\infty$. Suppose otherwise. Then there exist
$\varepsilon_0>0$ and sequences $\sigma_n,\tau_n\to\infty$ such that
\[
\left|
S(\sigma_n,\tau_n)
-
G\left(\frac{\tau_n}{\sigma_n}\right)
\right|
\geq\varepsilon_0
\]
for every $n$. Set
\[
r_n=\frac{\tau_n}{\sigma_n}.
\]
After passing to a subsequence, we may assume that $r_n\to \alpha$
for some $\alpha\in[0,\infty]$. By part \rm(ii),
\[
S(\sigma_n,\tau_n)
\to 
\sum_{k=1}^{\infty}
\bigl(\sqrt q\,\eta(\alpha)\bigr)^{-2k}
=
\frac{1}{q\eta(\alpha)^2-1}
=
G(\alpha).
\]
On the other hand, continuity of $G$ gives $
G(r_n)\to  G(\alpha)$,
which is a contradiction. Therefore
\[
S(\sigma,\tau)-G(\tau/\sigma)\to 0,
\]
which proves \rm(iii).
\end{proof}

We are now in a position to prove Theorem \ref{thm:thmA}, whose statement we repeat here for the reader's convenience.

\begin{theorem}\label{prop:heat_asymp_ell1}
		Let $h_t$ denote the heat kernel on a homogeneous tree. Then, for $x\in \mathcal{T}$ such that $|x|\to \infty$ and $t\to \infty$, we have the asymptotics
		\[
		h_t(x)\sim\frac{2}{\gamma(0)}
t^{-1}e^{-b_2t}
H\!\left(\frac{t\gamma(0)}{|x|+1}\right)
(1+|x|)q^{-|x|/2}
h_{t\gamma(0)}^{\mathbb Z}(|x|+1),\]
where $H(z)={q\eta(z)^2}/{(q\eta(z)^2-1)}$.
	\end{theorem}
	\begin{proof}
		By Proposition \ref{heatkernel}(ii), 
\begin{align}
\frac{\gamma(0)t}{2e^{-b_2t}}q^{|x|/2}h_t(x)&=
(|x|+1)h_{t\gamma(0)}^{\mathbb Z}(|x|+1) \notag \\
&\times\Bigg[
1+
\sum_{k=1}^\infty q^{-k}\,
\frac{|x|+2k+1}{|x|+1}\,
\frac{
h_{t\gamma(0)}^{\mathbb Z}(|x|+2k+1)
}{
h_{t\gamma(0)}^{\mathbb Z}(|x|+1)
}
\Bigg].
\label{eq:quotZheat1}
\end{align}

		On the one hand, we have
\[
\frac{|x|+2k+1}{|x|+1}=1+k\,O(|x|^{-1}).
\] On the other hand, as already mentioned (see for instance \cite[p.4282]{CMS}), the quotient of the heat kernels on $\mathbb{Z}$ in \eqref{eq:quotZheat1} is bounded by $1$ for all $|x|, t, k$. Since $kq^{-k}$ is summable, dominated convergence applies. 
We can then use Theorem \ref{thm:heat Z asym} for $h_t^{\mathbb{Z}}$ to get 
\[
h_\tau^{\mathbb Z}(j)=F(j,\tau)(1+o(1)), \quad \tau^2+j^2\to \infty.
\]
Moreover, this asymptotic can be used uniformly in $k$ in the present setting. Indeed, with $
\tau=t\gamma(0)$, $j=|x|+2k+1$,
we have
$j^2+\tau^2
\ge (|x|+1)^2+t^2\gamma(0)^2
\to \infty$
as $t,|x|\to\infty$, uniformly for all $k\ge0$. Hence
\[
\frac{h_{t\gamma(0)}^{\mathbb Z}(|x|+2k+1)}
{F(|x|+2k+1,t\gamma(0))}
=1+o(1),
\]
where the $o(1)$ is uniform in $k\ge0$. 

Next, for the quotient of $F$ functions,  we can invoke Lemma \ref{lem:F-ratios-corrected}(iii) for
\[
\sigma=|x|+1,
\qquad
\tau=t\gamma(0).
\]
Therefore, the quantity in brackets in \eqref{eq:quotZheat1} tends, as $t, |x|\to \infty$, to 
\[
1+\frac{1}{q\,\eta (t\gamma(0)/(|x|+1))^2-1}=
H\!\left(\frac{t\gamma(0)}{|x|+1}\right),
\]
which completes the proof.
\end{proof}

\begin{remark} Recall that $\gamma(0)=2\sqrt{q}/(q+1)$, and that by \eqref{eq:sphericalT}, we have
	\begin{align}\label{eq:phi0_asympt}
	\frac{(1+|x|)q^{-|x|/2}}{\varphi_0(x)}&=\frac{q+1}{q-1}+O(|x|^{-1}).
	\end{align}
Hence, the asymptotics above can be rewritten as
$$\frac{h_t(x)}
{t^{-1}e^{-b_2t}\varphi_0(x)
h_{t\gamma(0)}^{\mathbb Z}(|x|+1)}
\sim
\frac{(q+1)^2}{\sqrt q (q-1)}\,
H\left(\frac{t\gamma(0)}{|x|+1}\right),$$
which are compatible with the two-sided estimate of 
\cite[Proposition 2.5]{CMS} 	\begin{equation*}
\frac{q+1}{\sqrt{q}}\leq	\frac{h_t(x)}{t^{-1}{e^{-b_2t}}\, \varphi_0(x)\, h_{t\gamma(0)}^{\mathbb{Z}}(|x|+1) }\leq 
	\frac{\sqrt{q}(q+1)^3}{(q-1)^3},
    \end{equation*}
    since the quantity $H$ satisfies
\[
1
\leq
H(z)=\frac{q\eta(z)^2}{q\eta(z)^2-1}
\leq
\frac{q}{q-1},
\qquad z\in[0,\infty].
\]
\end{remark}
	
	\bigskip

	Finally, the next result concerns the long-time behavior of the heat kernel when it is spatially concentrated around the origin, and more precisely, when $|x| = o(\sqrt{t})$. This corresponds to the second part of Theorem \ref{thm:thmA}. Notice that $x$ need not travel to infinity, so this case is not covered by Theorem \ref{prop:heat_asymp_ell1}.
	
	\begin{proposition}\label{prop:heat_asymp_p>2} Let $x\in \mathcal{T}$ be such that $|x|\leq \rho(t)$, where $\rho(t)\to \infty$ and   $\frac{\rho(t)}{\sqrt{t}}\to 0$ as $t\to \infty.$
		Then
		\[
		h_t(x)=\frac{\sqrt{2}}{\sqrt{\pi}}\frac{q(q+1)}{(q-1)^2}\gamma(0)^{-3/2}\, e^{-b_2t}\,t^{-3/2}\, \varphi_0(x)\left(1+O\Big(\frac{\rho(t)}{\sqrt{t}}\Big)\right). 
		\]
	\end{proposition}
	
	We shall first need an auxiliary lemma.
	\begin{lemma}\label{lemma:Step5}
		For any $\delta>0$, we have 
		\[
		\int_{0}^{\tau/2} e^{-\delta t \sin^2(\lambda \frac{\log q}{2})}\,\sin^2(\lambda \log q)\, d\lambda= \frac{2\sqrt{\pi}}{\log q}\, \delta^{-3/2}\, t^{-3/2} +O(t^{-5/2}), \quad \text{as} \quad t\to \infty. 
		\]
	\end{lemma}
	\begin{proof}
		Observe first that for any $c>0$ and any $m\geq 0$, we have
		\begin{equation}\label{eq:auxint}
			\int_{0}^{\tau/2} e^{-ct\lambda^2}\, \lambda^{m}\, d\lambda \asymp_{c,m, \tau} t^{-\frac{m+1}{2}}, \quad t>1.
		\end{equation}
		Next, using that $\sin \alpha=2\sin\frac{\alpha}{2}\cos\frac{\alpha}{2} $ and the change of variables $\mu=\sin(\lambda \frac{\log q}{2})$, 
		\begin{align*}
			\int_{0}^{\tau/2} e^{-\delta t \sin^2(\lambda \frac{\log q}{2})}\,\sin^2(\lambda \log q)\, d\lambda&=\frac{8}{\log q}\int_{0}^{1} e^{-\delta t\mu^2}\,\mu^2\, \sqrt{1-\mu^2}\, d\mu\\
			&=\frac{8}{\log q}\int_{0}^{1} e^{-\delta t\mu^2}\,\mu^2\, (1+O(\mu^2)) \, d\mu.
		\end{align*}
		The error term is $O(t^{-5/2})$, owing to \eqref{eq:auxint}, so let us check the first integral; after the change of variables $\sqrt{\delta t}\mu=u$ we get
		\begin{align*}\frac{8}{\log q}\int_{0}^{1} e^{-\delta t\mu^2}\,\mu^2\, d\mu=\frac{8}{\log q}\delta^{-3/2}t^{-3/2}\int_{0}^{\sqrt{\delta t}}u^2\, e^{-u^2}\, du.
		\end{align*}
		Since
		\begin{align*}
			\int_{0}^{\sqrt{\delta t}}u^2\, e^{-u^2}\, du&=\int_0^{\infty}u^2\, e^{-u^2}\, du-\int_{\sqrt{\delta t}}^{\infty}u^2\, e^{-u^2}\, du\\
			&=\frac{\sqrt{\pi}}{4}+O(e^{-\delta t/2}),
		\end{align*}
		the proof is complete.
	\end{proof}
	
	\begin{proof}[Proof of Proposition \ref{prop:heat_asymp_p>2}]
		The proof follows ideas from \cite[Step 5]{AJ99}. Recall first that by the inversion formula \eqref{inversionT}, and \eqref{gammalambda0} for $\gamma(\lambda)$,
		\begin{align}
			\frac{2\pi(q+1)}{q\log q}\, h_t(x)&=\int_{0}^{\tau /2}e^{-t(1-\gamma(\lambda))}\varphi_{\lambda}(x)\frac{d \lambda}{|\textbf{c}(\lambda)|^2} \quad \notag\\
			&=e^{-b_2t}\, \varphi_0(x)\int_{0}^{\tau /2}e^{-t\gamma(0)(1-\cos(\lambda \log q))}\frac{d \lambda}{|\textbf{c}(\lambda)|^2} \notag\\
			&+e^{-b_2t}\int_{0}^{\tau /2}e^{-t\gamma(0)(1-\cos(\lambda \log q))}(\varphi_{\lambda}(x)-\varphi_0(x))\frac{d \lambda}{|\textbf{c}(\lambda)|^2} \notag \\
			&=I_t(x)+J_t(x). \label{eq:I+J}
		\end{align}	
		
        We will show that the main contribution comes from the term  $I_t(x)$.
		According to \eqref{eq:Planchereldens}, we can write 
		\[
		|\textbf{c}(\lambda)|^{-2}=4(q+1)^2\frac{\sin^2(\lambda \log q)}{b(\lambda)}, 
        \]
        where we put
        \[b(\lambda)=(q+1)^2\sin ^2(\lambda \log q)+(q-1)^2 \cos ^2(\lambda \log q).
		\]
		Write $1-\cos(\lambda \log q)=2\sin^2(\lambda\frac{\log q}{2})$ and split
		\begin{align*}
			I_t(x)&=4(q+1)^2\,\frac{1}{b(0)}\,e^{-b_2t}\, \varphi_0(x)\int_0^{\tau/2}e^{-2t\gamma(0)\sin^2(\lambda \frac{\log q}{2})}\sin^2(\lambda \log q)\, d\lambda \\
			&+4(q+1)^2\,e^{-b_2t}\ \varphi_0(x) \times\\
            &\times\int_0^{\tau/2}e^{-2t\gamma(0)\sin^2(\lambda \frac{\log q}{2})}\left(\frac{1}{b(\lambda)}-\frac{1}{b(0)}\right)\sin^2(\lambda \log q)\, d\lambda.
		\end{align*}
		One checks that $$b(\lambda)\asymp 1, \quad b(\lambda)-b(0)=O(\lambda),$$ since $b'$ is bounded. Then, by \eqref{eq:auxint}, the second integral in $I_t(x)$ is $O(t^{-2})$. On the other hand, the first integral can be computed using Lemma \ref{lemma:Step5} for $\delta=2\gamma(0)$;  so altogether, using that $b(0)=(q-1)^2$,
		\[
		I_t(x)=\frac{4\sqrt{\pi}}{\sqrt{2}}\frac{(q+1)^2}{\log q(q-1)^2}\gamma(0)^{-3/2}\,t^{-3/2}\, e^{-b_2t}\, \varphi_0(x)\,(1+O(t^{-1/2})).
		\]
		
		We now show that $J_t(x)$  is smaller than $I_t(x)$. 
		By the fundamental theorem of calculus,
\[
|\varphi_\lambda(x)-\varphi_0(x)|
\leq
\int_0^\lambda
\left|\frac{\partial\varphi_s}{\partial s}(x)\right|\,ds, \quad \lambda\in (0, \tau/2)
\]
where
$$\left|\frac{\partial\varphi_s}{\partial s}(x)\right|
\leq
(\log q)|x|
\int_\Omega q^{h_\omega(x)/2}\,d\nu(\omega)
=
(\log q)|x|\varphi_0(x),$$
		 by the integral representation formula \eqref{eq:integralrepspherical} and the bounds \eqref{eq:Busemann_bds}.
        
		Therefore, using that $1-\cos(\lambda \log q)\asymp \lambda^2$ and  $|\textbf{c}(\lambda)|^{-2}\lesssim \lambda^2$ on $[0, \tau/2]$ by \eqref{eq:Planchereldens}, we get for some $c>0$,
		\begin{align*}
			J_t(x)&\lesssim e^{-b_2t}\, \int_{0}^{\tau /2}e^{-t\gamma(0)(1-\cos(\lambda\log q))}\left|\varphi_{\lambda}(x)-\varphi_0(x)\right|\frac{d \lambda}{|\textbf{c}(\lambda)|^2}\\
			&\lesssim e^{-b_2t}\, |x|\, \varphi_0(x)\int_{0}^{\tau /2}e^{-ct\gamma(0)\lambda^2}\, \lambda^3\, d\lambda\\
			&\lesssim e^{-b_2t}\,t^{-3/2}\, \varphi_0(x)\, \frac{|x|}{\sqrt{t}},
		\end{align*}
		the last inequality owing to \eqref{eq:auxint}. Since $|x|\leq \rho(t)$ and $\frac{\rho(t)}{\sqrt{t}}\to 0$ as $t\to \infty$, we get
		$$J_t(x)=I_t(x)\, O\left(\frac{\rho(t)}{\sqrt{t}}\right),$$
		which proves the claimed heat kernel asymptotics  by \eqref{eq:I+J}.
	\end{proof}
	
\begin{remark} 
For asymptotics, one could work on the frequency side (as e.g. in \cite[Section 5]{AJ99} for symmetric spaces). More precisely, one could exploit the expansion of the spherical functions \eqref{eq:sphericalT}, use the explicit information for the Harish-Chandra $\textbf{c}$ function in \eqref{cfunc}, and work with oscillatory integrals; this would be closer in spirit to the last proposition. We prefer, however, to stick with the simpler approach of Theorem \ref{prop:heat_asymp_ell1} for the purposes of this work.
\end{remark}

	\subsection{Asymptotics for ratios of heat kernels}
	This subsection contains some technical results for ratios of heat kernels, which are to be used later. In fact, this subsection will be dealing with $|x|\to \infty$ as $t\to \infty$, with $|x|=O(t)$. To this end, let us start with a preparatory lemma. Already from the result that follows, it becomes clear that in this range, the asymptotic behavior of the quotient of the heat kernel depends on the (bounded) ratio $|x|/t$; moreover, the ratio limit is related to the \textit{Busemann} function.

   \begin{lemma}\label{lemma:quotient_prep}
    Let $x\in \mathcal{T}$ be such that $|x|=O(t)$, $|x|\to \infty$, and $t/|x|\to \alpha \in (0, \infty]$. Let also $y\in \mathcal{T}$ be bounded. Then
\[
\frac{h_t(d(x,y))}{h_t(d(x,o))}=q^{\frac{1}{2}(d(x,o)-d(x,y))}\, \eta(\alpha \gamma(0))^{d(x,o)-d(x,y)}(1+o(1)),
\]
The convergence is uniform for $y$ in any fixed bounded subset of $\mathcal T$.
    \end{lemma}

\begin{remark}
    For fixed $\alpha\in(0,\infty]$, due to 
$d(x,o)-d(x,y)=O(|y|)=O(1),
$
and \eqref{eq:def_eta},
the main term in the above expression remains bounded as $t\to\infty$.
\end{remark}

    \begin{proof}  The triangle inequality and the fact that $d(x,o)\to \infty$ as $t\to \infty$ yield $d(x,y)\asymp d(x,o)$ for $t$ large enough. 
    Using Theorem \ref{prop:heat_asymp_ell1} and Theorem \ref{thm:heat Z asym} for the asymptotics of $h_t$ and $h_t^{\mathbb{Z}}$ respectively, we can write as $t\to \infty$,
		\begin{align}\label{eq:quotientforL1}
			\frac{h_{t}(d(x,y))}{h_{t}(d(x,o))}&\sim q^{\frac{1}{2}(d(x,o)-d(x,y))}\,\frac{1+d(x,y)}{1+d(x,o)}\,  \frac{h_{t\gamma(0)}^{\mathbb{Z}}(d(x,y)+1)}{h_{t\gamma(0)}^{\mathbb{Z}}(d(x,o)+1)} \notag\\
            &\sim q^{\frac{1}{2}(d(x,o)-d(x,y))}\,\frac{1+d(x,y)}{1+d(x,o)}\,  \frac{F(d(x,y)+1, t\gamma(0))}{F(d(x,o)+1, t\gamma(0))}.
		\end{align}
		The triangle inequality implies that $$\frac{1+d(x,y)}{1+d(x,o)} =1+O(|x|^{-1}),$$ so it remains to study the asymptotics of the quotient of $F$ functions. Set 
        \[
        \sigma=d(x,o)+1,\quad m=d(x,y)-d(x,o), \quad \tau=t\gamma(0).
        \]
        Then $m$ is an integer satisfying $|m|\leq \rho$, while $\sigma, \tau \to \infty$, with $\tau/\sigma \to \alpha\gamma(0)$. Then Lemma \ref{lem:F-ratios-corrected} yields that 
        \begin{align*}
        \frac{F(d(x,y)+1, t\gamma(0))}{F(d(x,o)+1, t\gamma(0))}&=\eta(\alpha\gamma(0))^{-m}(1+o(1))\\
        &=\eta(\alpha\gamma(0))^{d(x,o)-d(x,y)}(1+o(1)),
        \end{align*}
        which proves the claim.
\end{proof}

The next two results focus on the regions that will be mostly relevant for the study of caloric functions. The size of $|x|$ compared to that of $t$ now becomes crucial.
	
    \begin{proposition}\label{ratio_asym_L1}
		Let $C_0>0$ and $r(t)$ be any function such that $r(t) / t^{1 / 2} \to\infty$ and  $r(t)/t\to 0$ as $t \to\infty$. Then for $x\in \mathcal{T}$ such that $ C_0t-r(t)\leq|x|\leq C_0t+r(t)$, and $y\in \mathcal{T}$ bounded, we have the asymptotics
		\[
		\frac{h_t(d(x,y))}{h_t(d(x,o))}=\left(\sqrt{q}\, \frac{1+\sqrt{1+s_0^2}}{s_0}\right)^{d(x,o)-d(x,y)} (1+o(1)),\]
		where $s_0=\frac{\gamma(0)}{C_0}$. The convergence is uniform for $y$ in any fixed bounded subset of $\mathcal T$.
	\end{proposition}
	
	\begin{proof}
    This is an immediate consequence of Lemma \ref{lemma:quotient_prep}, the fact that $t/|x|\to 1/C_0>0$ in this region, and the formula for $\eta$ in \eqref{eq:def_eta}.
	\end{proof}
	
	\begin{remark}\label{remark:q^{1/p}} Let $C_0=R_p=\frac{q^{1/p}-q^{1/p'}}{q+1}$, $1\leq p< 2$, be the constants of Theorem \ref{prop:MedSe_ellp} (i). Set
    $$
		\delta_p=\frac{1}{p}-\frac{1}{2}.
		$$
    Since $$q^{1 / p}-q^{1 / p^{\prime}}=q^{1 / 2}\left(q^{\delta_p}-q^{-\delta_p}\right) =2 \sqrt{q} \sinh (\delta_p \ln q),$$
   it follows that 
   \[
    s_0=s_0(p)=\frac{\gamma(0)}{R_p}= \frac{1}{\sinh (\delta_p\ln q)}.  \]
		Thus, we get
		\begin{align*}
			\sqrt{q}\, \frac{1+\sqrt{1+s_0(p)^2}}{s_0(p)}&=\sqrt{q}(\sinh (\delta_p \ln q)+\cosh (\delta_p \ln q))\\
			&=q^{1/2}\, e^{\delta_p \ln q}\\
			&=q^{1/p}.
		\end{align*}
	\end{remark}

	\begin{proposition}\label{prop:ratio_asym_L2}
		Assume $r_1(t), r_2(t)$ are positive functions such that \[
        \begin{cases}
        \frac{r_1(t)}{\log t}\to \infty \quad \text{and} \, &\frac{r_1(t)}{t^{1 / 2}} \to 0, \\
        \frac{r_2(t)}{t^{1/2}} \to \infty \quad \text{and} \, &\frac{r_2(t)}{t}\to 0,
        \end{cases}
        \]
        as $t \to\infty$. 
        Then for $x\in \mathcal{T}$ such that $r_1(t)\leq|x|\leq r_2(t)$, and $y$ bounded,
		\begin{align*}
			\frac{h_t(d(x,y))}{h_t(d(x,o))}&= \frac{\varphi_0(d(x,y))}{\varphi_0(d(x,o))}\, (1+o(1))\\
			&=q^{\frac{1}{2}(d(x,o)-d(x,y))}\, (1+o(1)).	
            \end{align*}
		The convergence is uniform for $y$ in any fixed bounded subset of $\mathcal T$.
	\end{proposition}
    
	\begin{proof}
    In the present time-space region, we have $$|x|\geq r_1(t)\to \infty, \quad  t/|x|\geq t/r_2(t)\to \infty.$$ Thus, since $\eta(\infty)=1$, and $y$ is bounded, an application of Lemma \ref{lemma:quotient_prep} yields
    \[
    \frac{h_t(d(x,y))}{h_t(d(x,o))}=q^{\frac{1}{2}(d(x,o)-d(x,y))}(1+o(1)).
    \]
		For the expression with the ground spherical functions, we have
		\[
		\frac{1+c_q\,d(x,y)}{1+c_q\,d(x,o)}= 1+O(r_1(t)^{-1}), \quad c_q=\frac{q-1}{q+1},
		\]
		so by \eqref{eq:sphericalT},
	        \begin{align*}
			\frac{\varphi_0(d(x,y))}{\varphi_0(d(x,o))}
			&= q^{\frac{1}{2}(d(x,o)-d(x,y))}(1+O(r_1(t)^{-1})),
            \end{align*}
	which proves the claim.
	\end{proof}

	\section{Long-time behavior of solutions to the heat equation on a homogeneous tree} \label{sec:T}

This section is devoted to the large-time behavior of caloric functions $e^{-t\mathcal{L}}f$ on a homogeneous tree, provided $f$ belongs to certain classes of data. To highlight the role of geometry, we first compare with the integer lattice $\mathbb{Z}$, where the behavior is analogous to the continuous Euclidean setting. Namely, if $f\in \ell^1(\mathbb{Z})$, $M=\sum_{j\in\mathbb{Z}}f(j)$, and
\[
u(t;j)=\sum_{k\in\mathbb{Z}}h_t^{\mathbb{Z}}(j-k)f(k),
\]
then, for every $1\le p<\infty$,
\[
\|h_t^{\mathbb{Z}}\|_{\ell^p(\mathbb{Z})}^{-1}
\|u(t;\cdot)-Mh_t^{\mathbb{Z}}\|_{\ell^p(\mathbb{Z})}
\to 0,
\qquad t\to\infty.
\]
Thus, as in the Euclidean setting, the long-time profile is completely determined by the total mass of the initial datum. The aforementioned result is proved in \cite[Theorem~6.1]{Abadiasetal}.
	
	In contrast, on homogeneous trees, there is no \textit{constant} $M$ that works for \textit{all} $p$. In this section, for \textit{each} $p \in [1, \infty]$, we shall introduce a notion 
	of a $p$-mass \textit{function} and prove that caloric functions with finitely supported initial data,  asymptotically decouple as the product of this mass function and the heat kernel in the $\ell^p$ sense. This function boils down to a constant, still depending on $p$, if the initial datum is radial.  Interestingly, the same object was the correct mass function for isotropic finite-range aperiodic random walks on homogeneous trees, see \cite[Subsection 4.2]{PT}.

    \subsection{Finitely supported initial data.}\label{sec:finitelysupp}
	
	This subsection deals with the large-time behavior of caloric functions $e^{-t\mathcal{L}}f$ on a homogeneous tree, provided $f$ is finitely supported. More precisely, in the rest of this subsection we prove Theorem \ref{thm:thmB} for compactly supported initial data, first for $1\leq p<2$, then for $2<p\leq \infty$, and finally we discuss $p=2$.
	
	\subsubsection{The case $1\leq p< 2$.} Let $f$ be finitely supported, and consider the function
	\begin{align}\label{eq:massfct1to2}
		M_pf(x)&=\sum_{y\in \mathcal{T}} f(y) \,\frac{1}{\nu(\Omega(o,x))} \int_{\Omega(o,x)}q^{\frac{1}{p}h_{\omega}(y)}\, d\nu(\omega) \notag\\
		&=\sum_{y\in \mathcal{T}} f(y) \fint _{\Omega(o,x)}q^{\frac{1}{p}h_{\omega}(y)}\, d\nu(\omega).
	\end{align} 
	Clearly, this function is bounded, since for any $\omega\in \Omega$ and $y\in \mathcal{T}$, we have $|h_{\omega}(y)|\leq |y|$ by \eqref{eq:Busemann_bds}.

	Moreover, if $f$ is radial, then this mass function becomes a constant, namely
	$$M_pf(x)=M_pf=\mathcal{H}f(\pm i\delta_p)=\sum_{y\in \mathcal{T}}f(y)\varphi_{\pm i\delta_p}(y), \quad \forall x\in \mathcal{T},$$
	where $\delta_p$ is as in \eqref{eq:deltap}. Indeed, recalling the definition of the Helgason-Fourier transform \eqref{eq:HFtransform}, and that it boils down to the spherical transform for radial functions, we have
	\begin{align*}
		\sum_{y\in \mathcal{T}} f(y) \fint _{\Omega(o,x)}q^{\frac{1}{p}h_{\omega}(y)}\, d\nu(\omega)&= \fint _{\Omega(o,x)}\sum_{y\in \mathcal{T}} f(y)q^{\frac{1}{p}h_{\omega}(y)}\,  d\nu(\omega)\\
		&=\fint _{\Omega(o,x)}\widehat{f}(-i\delta_p, \omega)\,  d\nu(\omega) \\
		&=\mathcal{H}f(\pm i\delta_p)\fint _{\Omega(o,x)}d\nu(\omega)\\
		&=\mathcal{H}f(\pm i\delta_p).
	\end{align*}
	
	Notice that for $p=1$ and $f$ radial, we get by \eqref{eq:integralrepspherical} that $\varphi_{i/2}\equiv 1$, hence
	\[
	M_1(f)=\sum_{x\in \mathcal{T}}f(x),
	\]
	which is reminiscent of the mass (constant) in the case of $\mathbb{Z}$.

	Let $\supp f \subseteq B(o,\rho)$ for some $\rho>0$. We discuss the behavior of caloric functions 
	\[
	u(t;x)=e^{-t\mathcal{L}}f(x)=f\ast h_t(x)= \sum_yf(y) \, h_t(d(x,y)),
	\]
	 separately inside and outside the critical $\ell^p$ region $\textbf{B}_p(t)$, $1\leq p <2$, whose description is given in Theorem \ref{prop:MedSe_ellp}.

	\subsubsection{Outside the critical $\ell^p$ region}\label{subsec:outside}
	The next result shows that the solution itself vanishes asymptotically in time in the $\ell^p$ sense, provided we sum outside the critical region $\textbf{B}_p(t)$, $1\leq p <2$, which was defined in Theorem \ref{prop:MedSe_ellp}. 
	\begin{proposition}\label{prop:outsideL1crit}
		Assume that $\supp f \subseteq B(o,\rho)$ for some $\rho>0$. Then the corresponding caloric function $u=e^{-t\mathcal{L}}f$ satisfies
		\[
\|u(t;\cdot)\|_{\ell^p(\mathcal{T}\setminus \textbf{B}_p(t))}=o(\|h_t\|_{\ell^p}), \quad \text{as}\quad t \to \infty.
		\]
	\end{proposition}
	\begin{proof}
		By the triangle inequality, if $d(x,o)>R_pt+r(t)$, $R_p>0$, then 
		\[d(x,y)\geq d(x,o)-d(y,o)>R_pt+r(t)-\rho>R_pt+\frac{1}{2}r(t),\]
		while if $d(x,o)<R_pt-r(t)$, then 
		\[ d(x,y)\leq d(x,o)+d(y,o)< R_pt-r(t)+\rho< R_pt-\frac{1}{2}r(t),
		\]
		for $t$ large enough.  Therefore, if $x\notin \textbf{B}_p(t)$, then $x\notin \widetilde{\textbf{B}}_{p}^{y}(t)$, where we define
		\begin{equation*}
		\widetilde{\textbf{B}}_{p}^{y}(t)=\left\{x\in \mathcal{T}:\, |d(x,y)-R_pt|\leq\frac{1}{2}r(t)\right\}.
		\end{equation*}
		It follows that \begin{align*}\|u(t;\cdot)\|_{\ell^p(\mathcal{T}\setminus \textbf{B}_p(t))}&=\bigg(\sum_{x\in \mathcal{T}\setminus \textbf{B}_p(t)} 
		\big| \sum_{|y|<\rho} h_t(d(x,y)) \, f(y) \big|^p \bigg)^{\frac{1}{p}} \notag\\
		&\leq 
		\sum_{|y|<\rho}
		\bigg( \sum_{x\in \mathcal{T}\setminus \textbf{B}_p(t)} h_t(d(x,y))^p\, |f(y)|^p \bigg)^{\frac{1}{p}} \notag\\
		&=
		\sum_{|y|<\rho}|f(y)|\bigg( \sum_{x\in \mathcal{T}\setminus \textbf{B}_p(t)} h_t(d(x,y))^p \bigg)^{\frac{1}{p}}.
	\end{align*}
	Using the inclusion $\mathcal{T}\setminus \textbf{B}_{p}(t)\subseteq \mathcal{T}\setminus\widetilde{\textbf{B}}_{p}^{y}(t)$ for $t$ sufficiently large, we get
        
    \begin{align*}
		\|u(t;\cdot)\|_{\ell^p(\mathcal{T}\setminus \textbf{B}_p(t))}
		&\leq
		\sum_{|y|<\rho}|f(y)|\bigg( \sum_{x\notin \widetilde{\textbf{B}}_{p}^{y}(t)} h_t(d(x,y))^p \bigg)^{\frac{1}{p}}.
		\end{align*}
		Using Theorem \ref{prop:MedSe_ellp}(i) for $h_t(d(\cdot, y))$ over the complement of $\widetilde{\textbf{B}}_{p}^{y}(t)$, we can conclude. 
	\end{proof}

	As a result, since the mass function is bounded, we have
	\begin{align}
	\|h_t\|_p^{-1}\, \|e^{-t\mathcal{L}}f  &- M_pf\cdot h_t\|_{\ell^p(\mathcal{T}\setminus \textbf{B}_p(t))}
		\leq \|h_t\|_p^{-1}
		\|u(t; \cdot)\|_{\ell^p(\mathcal{T}\setminus \textbf{B}_p(t))} \notag\\
		&+\|M_pf\|_{\infty}\,\|h_t\|_p^{-1}
		\|h_t\|_{\ell^p(\mathcal{T}\setminus \textbf{B}_p(t))}\to 0, \label{eq:conv_out_ell1}
	\end{align}
	owing to  Proposition \ref{prop:outsideL1crit} and Theorem \ref{prop:MedSe_ellp}.
	
	\subsubsection{Inside the critical region.} 	For $x \in \textbf{B}_p(t)$, we write  for $u(t;x)-M_pf(x)\,h_t(x)$,
	\begin{align}
		&\sum_{y\in \mathcal{T}} h_t(x,y) \,f(y)-M_pf(x)\, h_t(x) \notag \\
		&\qquad
		=h_t(x)
		\sum_{y:\, d(y,o)<\rho}
		f(y) \left( \frac{h_t(d(x,y))}{h_t(d(x,o))}-\fint_{\Omega(o,x)}q^{\frac{1}{p}h_{\omega}(y)}\, 
		d\nu(\omega) \right). \label{eq:ell1diff}
	\end{align}
	
    Using the quotient asymptotics from Proposition \ref{ratio_asym_L1}, and Remark \ref{remark:q^{1/p}}, we obtain
	\begin{align*}
		\frac{h_t(d(x,y))}{h_t(d(x,o))}=q^{\frac{1}{p}(d(x,o)-d(x,y))}(1+o(1))=q^{\frac{1}{p}(d(x,o)-d(x,y))}+o(1),
	\end{align*}
	since $d(x,o)-d(x,y)$ is bounded.  For each $x\in \textbf{B}_p(t)$, each $\omega \in \Omega(o, x)$, and $t$ sufficiently large,  we have
	$x \in [y, \omega)$ because $d(x,o)\to \infty$; hence
	\[
	h_{\omega}(y)=d(o,x)-d(y, x), \quad \omega \in \Omega(o,x), \quad |x|\to \infty, \quad |y|<\rho.
	\]
	Therefore,
	\begin{align*}
		\fint_{\Omega(o,x)} q^{\frac{1}{p}h_{\omega}(y)} \, d\nu(\omega)
		&=\fint_{\Omega(o,x)}q^{\frac{1}{p}(d(o, x)-d(y, x))} \, d\nu(\omega) \\
		&=q^{\frac{1}{p}(d(o, x)-d(y, x))}.
	\end{align*}
    Since $\supp f$ is finite, the preceding asymptotics are uniform for $y \in \supp f$.
Consequently, there exists $\varepsilon(t) \to 0$ such that, uniformly for $x\in \textbf{B}_p(t)$,  equation \eqref{eq:ell1diff} yields
	\[
	|u(t; x) - M_pf(x)\, h_t(x)|
	\leq 
	\varepsilon(t)\, h_t(x) \|f\|_{\ell^1}.
	\]
	 Summing over the critical region $\textbf{B}_p(t)$, $1\leq p<2$, we get
	\begin{align*}
	\left(	\sum_{x \in \textbf{B}_p(t)}
		\left|u(t; x) - M_pf(x)\, h_t(x)\right|^p \right)^{1/p}&\leq
		\varepsilon(t)\,\bigg(\sum_{ \textbf{B}_p(t)} h_t^p\bigg)^{1/p} \|f\|_{\ell^1}\\
		&\leq \varepsilon(t)\,
		\|h_t\|_{\ell^p} \|f\|_{\ell^1},
	\end{align*}
	which completes the proof of Theorem \ref{thm:thmB} in the $\ell^p$ case, $1\leq p<2$. Indeed, combining this estimate with \eqref{eq:conv_out_ell1}, we have shown that
	\[
		\|h_t\|_{\ell^p}^{-1}\|e^{-t\mathcal{L}}f 
	- M_pf\cdot h_t\|_{\ell^p}\to 0, \quad \text{as} \quad t \to \infty.
	\] 
	
	\subsubsection{The case $2<p\leq \infty$}
	Let us start with the following observation: fix $y\in \mathcal{T}$; then for all $x\in \mathcal{T}$, 
	by the integral representation  \eqref{eq:integralrepspherical0} of the ground spherical function,
	\begin{align}
		\varphi_{0}(d(x,y))
		&= 
		\int_\Omega q^{\frac{1}{2}h_{\omega}(y)}\,
		 q^{\frac{1}{2}h_{\omega}(x)}   d\nu( \omega) \notag\\
		&\leq q^{\frac{1}{2}|y|}\,\varphi_0(x),\label{eq:phi_quot_bdd}
	\end{align}
	by the Busemann function bounds \eqref{eq:Busemann_bds}.
	
	Recall now the $\ell^p$ critical region, $2<p\leq \infty$, defined in Theorem \ref{prop:MedSe_ellp} (we assume in addition that $r_3(t)$ is slow at infinity):  
	$$\textbf{B}_p(t)=\{x\in \mathcal{T}:\, d(x,o)\leq r_3(t)\},  \quad \frac{r_3(t)}{\sqrt{t}}\to 0 \quad \text{as} \quad t\to \infty.$$
	We work for $u(t;x)=e^{-t\mathcal{L}}f$, with $\supp f\subseteq B(o, \rho)$, by distinguishing cases whether $x$ belongs to $\textbf{B}_p(t)$ or not. 
	
	Outside the critical region, by the Minkowski inequality we have
	\begin{align*}
		\|e^{-t\mathcal{L}}f\|_{\ell^p(\mathcal{T} \setminus \textbf{B}_p(t))}&=\Bigg(\sum_{x \in \mathcal{T} \setminus \textbf{B}_p(t)} 
		\big| \sum_{y: \, d(o,y)<\rho} h_t(d(x,y)) \, f(y) \big|^p \bigg)^{\frac{1}{p}}\\
		&\leq 
		\sum_{y: \, d(o,y)<\rho}
		\bigg( \sum_{x \in \mathcal{T} \setminus \textbf{B}_p(t)} h_t(d(x,y))^p\, |f(y)|^p \bigg)^{\frac{1}{p}} \notag\\
		&=
		\sum_{y: \, d(o,y)<\rho}|f(y)|\bigg( \sum_{x \in \mathcal{T} \setminus \textbf{B}_p(t)} h_t(d(x,y))^p \bigg)^{\frac{1}{p}}. 
	\end{align*}
For $x\notin \textbf{B}_p(t)$ and $y$ bounded, we have
	\[
	d(x,y)\geq r_3(t)-d(y,o)\geq \frac{1}{2}r_3(t),
	\] 
for $t$ large enough; therefore, the inner sum above is dominated by the sum in $\mathcal{T}\setminus B(y, \frac{1}{2}r_3(t))$. In turn, this is $o(\|h_t\|_p)$ by a standard modification of the arguments of Proposition \ref{prop:outsideL1crit} and using Theorem \ref{prop:MedSe_ellp}. It follows that
\[
\|h_t\|_p^{-1}\,\|u(t;\cdot)\|_{\ell^p(\mathcal{T}\setminus \textbf{B}_p(t))}\to 0.
\]

Instead, when $x$ is in the critical region $\textbf{B}_p(t)$, and $y$ is bounded, the result of Proposition \ref{prop:heat_asymp_p>2} yields
	\begin{align}\label{eq:ratio_p>2}
		\frac{h_t(d(x,y))}{h_t(x)}&=\frac{\varphi_0(d(x,y))}{\varphi_0(x)} (1+o(1))\notag\\
		&=\frac{\varphi_0(d(x,y))}{\varphi_0(x)} +o(1).
	\end{align}
where the last equality is justified by \eqref{eq:phi_quot_bdd}.

	For $2< p\leq \infty$, define now the mass function
	\begin{equation}\label{eq:massfct>2}
	M_pf(x)=\sum_{y\in \mathcal{T}}\frac{\varphi_0(d(x,y))}{\varphi_0(x)}f(y)=\frac{1}{\varphi_0(x)}f\ast\varphi_0(x), \quad x\in \mathcal{T}.
	\end{equation}
	Hence, owing to \eqref{eq:phi_quot_bdd}, when $f$ is finitely supported, the mass function is bounded. Moreover, when $f$ is radial, the mass function boils down to the constant $\mathcal{H}f(0)$, since $f\ast\varphi_0(x)=\mathcal{H}f(0)\, \varphi_0(x)$.

	Thus for $x\in \textbf{B}_p(t)$, $2<p\leq \infty$, we have for $u(t;x)-M_pf(x)\,h_t(x)$,
	\begin{align*}
		&\sum_{y\in \mathcal{T}} h_t(d(x,y))\,f(y) - M_pf(x)\, h_t(x) \\
		&\qquad\qquad=
		h_t(x)
		\sum_{y:\,  d(o,y)< \rho} f(y) 
		\left( \frac{h_t(d(x,y))}{h_t(x)} -  \frac{\varphi_0(d(x,y))}{\varphi_0(x)} \right).
	\end{align*}
Since $\supp f$ is finite, the asymptotics of  \eqref{eq:ratio_p>2}, are uniform for $y \in \supp f$.
Consequently, there exists $\epsilon(t) \to 0$ such that, uniformly for $x\in \textbf{B}_p(t)$,  equation \eqref{eq:ratio_p>2} yields
	\[
	|u(t; x) - M_pf(x)\, h_t(x)|
	\leq \epsilon(t)\,
	h_t(x) \|f\|_{\ell^1}.
	\]
	Hence,  if $p \in (2, \infty)$, by summing over the critical region $\textbf{B}_p(t)$, we get
\begin{align*}
	\left(	\sum_{x \in \textbf{B}_p(t)}
	\left|u(t; x) - M_pf(x)\, h_t(x)\right|^p \right)^{1/p}&\leq \epsilon(t)\,
	\bigg(\sum_{x \in \textbf{B}_p(t)} h_t(x)^p\bigg)^{1/p} \|f\|_{\ell^1} \\
	&\leq \epsilon(t)\,
	\|h_t\|_{\ell^p} \|f\|_{\ell^1},
	\end{align*}
	which shows that
	\[
	\|h_t\|_{\ell^p}^{-1}\|u(t; \cdot)-M_pf(\cdot)\, h_t(\cdot)\|_{\ell^p(\textbf{B}_p(t))}\to 0, \quad \text{as} \quad t\to \infty.
	\]
	We have thus shown that for $2<p< \infty$, if the initial datum $f$ is finitely supported and $M_pf$ is defined as in \eqref{eq:massfct>2}, then for  $u=e^{-t\mathcal{L}}f$ we have
	\[
	\|h_t\|_{p}^{-1}\|e^{-t\mathcal{L}}f-M_pf\cdot h_t\|_{p}\to 0, \quad \text{as} \quad t\to \infty.
	\]
	
    Lastly, for $p = \infty$ the reasoning is similar. We omit the details.
	
	\subsubsection{The case $p=2$} Using the arguments of the previous two sections, it should be clear that when $f$ is finitely supported, we can use both expressions 
    \begin{equation}\label{eq:M2}
    \varphi_0^{-1}(x)(f\ast\varphi_0)(x) \quad  \text{and}  \quad \sum_{y\in \mathcal{T}}f(y)\fint_{\Omega(o,x)}q^{\frac{1}{2}h_{\omega}(y)}\, 
	d\nu(\omega),
    \end{equation}
    to serve as a mass function $M_2f$; notice that both expressions in \eqref{eq:M2} are bounded functions. 

		The reason for that is the heat kernel quotient asymptotics of Proposition \ref{prop:ratio_asym_L2}, in the $\ell^2$ critical region $\textbf{B}_2(t)$ (we assume in addition to Theorem \ref{prop:MedSe_ellp} that $r_1(t)/\log t\to \infty$ and $r_2(t)/t\to 0$). Starting from the asymptotics $h_t(d(x,y))/h_t(x)\sim \varphi(d(x,y))/\varphi(x)$, using  the expression $M_2f=\varphi_0^{-1}(f\ast\varphi_0)$ and working as in the case $2<p\leq \infty$, one can show that 
\begin{equation}\label{eq:ell2_caloric_asympt}
		\|h_t\|_{\ell^2}^{-1}\|e^{-t\mathcal{L}}f-M_2f\cdot h_t\|_{\ell^2}\to 0 \quad \text{as} \quad  t\to \infty.
		\end{equation}
		Alternatively, notice that $x\in \textbf{B}_2(t)$ implies $d(x,o)\to \infty$. Thus, using in Proposition \ref{prop:ratio_asym_L2} the asymptotics $h_t(d(x,y))/h_t(x)\sim q^{\frac{1}{2}(d(x,o)-d(x,y))}$, one can use the second expression in \eqref{eq:M2} as $M_2f$, then argue as in the case $1\leq p<2$  to get \eqref{eq:ell2_caloric_asympt}. 
		
		It needs to be stressed that the quantities in \eqref{eq:M2} are not necessarily equal everywhere, as one can see by taking $f=\delta_y$, $y\neq o$: the first expression simply becomes $\varphi_0(d(x,y))/\varphi_0(x)$ and the second $\fint_{\Omega(o,x)}q^{\frac{1}{2}h_{\omega}(y)}\, 
d\nu(\omega)$. However, they are \emph{asymptotically} equivalent. Indeed, in Proposition \ref{prop:ratio_asym_L2}, we have shown that
	\begin{align*}
		\frac{\varphi_0(d(x,y))}{\varphi_0(d(x,o))}
		&= q^{\frac{1}{2}(d(x,o)-d(x,y))}(1+O(r_1(t)^{-1}))\\
		&=q^{\frac{1}{2}(d(x,o)-d(x,y))}+O(r_1(t)^{-1}),\quad x\in \textbf{B}_2(t), \quad y \, \text{bounded}.
	\end{align*}
	 In the case that $f$ is radial, though,  both expressions boil down (globally) to the same constant $\mathcal{H}f(0)$.

	Hence, by this slight abuse of notation, that is, calling them both $M_2f$, we wish to emphasize that $p=2$ is a critical value for the change of the behavior of mass functions. It is exactly for this value of $p$ that the exponential volume growth of the graph is canceled by the exponential decay of $\varphi_0^2$ in $h_t^2$ (recall the heat kernel estimates in Proposition \ref{heatkernel}).  Thus, the case $p=2$ is a threshold for the behavior of mass functions and thus the asymptotic behavior of solutions to the heat equation.

\subsection{Beyond finitely supported data in $\mathcal{T}$}
\subsubsection{Radial initial data.} The aim of this section is to go beyond finitely supported initial conditions. Let us first start with the spaces $\ell^1_{\delta_p}(\mathcal{T})$, $p \in [1, \infty]$, which are defined  in \eqref{eq:defdeltap}.
Let us observe that $\ell^1_{1/2}(\mathcal{T}) = \ell^1(\mathcal{T})$, since by \eqref{eq:integralrepspherical},  $\varphi_{i/2}\equiv 1$. Moreover, if $f$ belongs to $\ell^1(\mathcal{T})$ then it belongs
to all $\ell^1_{\delta_p}(\mathcal{T})$ for $p \in (1, \infty]$. This follows by \eqref{eq:sphericalT}, which implies that $\varphi_0 \lesssim 1$ while $0<\varphi_{i\delta_p}\lesssim_{p,q} q^{-{|\cdot|}/{p'}}$ for all $p\in (1,2)$.

Furthermore, we can extend the definition of the mass function to radial functions belonging to $\ell^1_{\delta_p}(\mathcal{T})$. 
In fact, in this case, the mass is a constant, the reasoning being the same as in the finitely supported case for radial initial conditions, and the definition of the spherical transform, see Section \ref{sec:FT}. More precisely, for all $x\in \mathcal{T}$
\[
M_pf(x)=M_pf=\begin{cases}
    \mathcal{H} f(i\delta_p), \quad &p \in [1, 2];\\
    \mathcal{H} f(0), \quad &p \in [2, \infty].
\end{cases}
\]

Our aim in this subsection is to extend the previous result to \emph{radial} functions in $\ell^1_{\delta_p}(\mathcal{T})$, that is, to prove Theorem \ref{thm:thmB} for such initial data.
The proof draws on both the Kunze--Stein phenomenon and Herz’s \emph{principe de majoration} for radial convolutors, tools that are available on homogeneous trees. We recall them briefly, but for more information, we refer the reader to \cite{CGM, CMS0, CMS, N}.

We denote by $\ell^p_{\text{rad}}(\mathcal{T})$ the subset of $\ell^p(\mathcal{T})$ constisting of radial functions.
The Kunze-Stein phenomenon states that
for all $1 \leq p < 2$, we have
$$\ell^2(\mathcal{T}) \ast \ell^p_{\text{rad}}(\mathcal{T}) \subseteq \ell^2(\mathcal{T}), $$
see e.g. \cite[Theorem 1.4(ii)]{CMS}. Duality and commutativity of convolution for radial functions then yield 
\begin{equation}\label{eq:KS}
\ell^2_{\mathrm{rad}}(\mathcal{T})
\ast
\ell^2_{\mathrm{rad}}(\mathcal{T})
\subseteq
\ell^s_{\mathrm{rad}}(\mathcal{T}),
\qquad 2<s\leq\infty.
\end{equation}

Herz’s \textit{principe de majoration} ensures that a positive radial function
$k$ convolves $\ell^p(\mathcal{T})$, $p\in [1,2]$, into itself if and only if $\mathcal{H}k(i\delta_p)$ is finite, see e.g. \cite{VECA}, and in fact $$\|\ast k\|_{\ell^p(\mathcal{T})\to \ell^p(\mathcal{T})}=\mathcal{H}k(i\delta_p).$$

The following result allows us to consider the whole class of radial functions in $\ell^1_{\delta_p}(\mathcal{T})$ as initial conditions for the $\ell^p$ asymptotic convergence of caloric functions. It follows some ideas on symmetric spaces presented in \cite{Naik}, see also \cite{P1} or \cite{PT}. The arguments are relatively standard once the convergence result for finitely supported data has been established, but we include them for the sake of completeness. 
\begin{theorem}
	\label{thm:9}
	Let $p \in [1, \infty]$. Then, if $f$ is a radial function
	in $\ell^1_{\delta_p}(\mathcal{T})$, we have
	\[
	\lim_{t \to \infty} 
	\frac{1}{\|h_t \|_{\ell^p}}
	\big\|e^{-t\mathcal{L}}f-M_pf \cdot h_t\|_{\ell^p} = 0.
	\]
\end{theorem}
\begin{proof}
	First, let us consider $p \in [1, 2]$. We employ a density argument to use the convergence result we already had if the initial condition was finitely supported. 
	
	Assume $f \in \ell^1_{\delta_p}(\mathcal{T})$ is a radial function. Let $\varepsilon > 0$. Let $\tilde{f}$ be a radial and finitely supported 
	function such that
	\[
	\sum_{y \in \mathcal{T}} |f(y) - \tilde{f}(y)| \varphi_{i\delta_p}(y) \leq \tfrac{1}{3} \varepsilon.
	\]
	Since
	\begin{align*}
		\big| M_pf-M_p\tilde{f} \big| 
		&\leq 
		\sum_{y \in \mathcal{T}} \big|f(y) - \tilde{f}(y)\big| \varphi_{i\delta_p}(y) = \mathcal{H}(|f - \tilde{f}|)(i\delta_p) ,
	\end{align*}
	we have
	\[
	|M_pf-M_p\tilde{f}|\, \|h_t\|_{\ell^p} \leq \tfrac{1}{3} \varepsilon \|h_t\|_{\ell^p}.
	\]
	Let us set $u=e^{-t\mathcal{L}}{f}$, $\tilde{u}=e^{-t\mathcal{L}}\tilde{f}$.
	Then by the results of Section \ref{sec:finitelysupp} for finitely supported initial conditions, there is $t_1 \geqslant 1$ such that for all $t \geqslant t_1$,
	\[
	\frac{1}{\|h_t\|_{\ell^p}} \,\|\tilde{u}(t; \cdot)-M_p\tilde{f} \cdot h_t(\cdot) \|_{\ell^p}
	\leq \tfrac{1}{3} \varepsilon.
	\]
	Since $h_t$ is radial, and so are $f$ and $\tilde{f}$, their convolution is commutative. By the radial version of Herz's \emph{principe},  we get
	\[
	\big\|u(t;\cdot)-\tilde{u}(t;\cdot)\big\|_{\ell^p}
	\leq
	\mathcal{H} (|f - \tilde{f}|)(i\delta_p) \|h_t\|_{\ell^p} 
	\leq 
	\tfrac{1}{3} \varepsilon \|h_t\|_{\ell^p}.
	\]
	Therefore, for $t \geqslant t_1$, we obtain
	\begin{align*}
		\big\|u(t; \cdot) - M_pf \cdot h_t(\cdot)\big\|_{\ell^p}
		&\leq
		\big\|u(t;\cdot)-\tilde{u}(t;\cdot)\big\|_{\ell^p}
		+|M_pf-M_p\tilde{f}|\, \|h_t\|_{\ell^p} \\
		&\phantom{\leq \big\|u(t;\cdot)-\tilde{u}(t;\cdot)\big\|_{\ell^p} \ }
		+\big\|\tilde{u}(t;\cdot) - M_p\tilde{f}\, h_t(\cdot) \|_{\ell^p} \\
		&\leq \varepsilon \|h_t\|_{\ell^p},
	\end{align*}
	which finishes the proof.
	
Assume that $p \in (2, \infty]$ and let $t>0$. Since by the semigroup property $h_t=e^{-(t/2)\mathcal{L}}h_{t/2}$, that is,
\[
h_t(d(x,y)) =\sum_{z \in \mathcal{T}} h_{t/2}(d(x,z)) h_{t/2}(d(y,z)),
\]
we have
\begin{align*}
	u(t; x) 
	&= \sum_{y \in \mathcal{T}} f(y) h_t(d(x,y)) \\
	&= \sum_{y \in \mathcal{T}} f(y) \sum_{z \in \mathcal{T}} h_{t/2}(d(y, z)) h_{t/2}(d(z,x)) \\
	&= \sum_{z \in \mathcal{T}} u(t/2; z) h_{t/2}(d(z,x)).
\end{align*}
Hence,
\begin{align*}
	u(t; x) - M_pf  \cdot h_t(x) = \sum_{z \in \mathcal{T}} \big( u(t/2; z) - M_pf \cdot h_{t/2}(z)\big) h_{t/2}(d(z,x)).
\end{align*}
Therefore, by \eqref{eq:KS},  we get 
\begin{align*}
	\left\| u(t; \cdot) - M_pf  \cdot h_t \right\|_{\ell^p}
	&=
	\left\|  \sum_{z \in \mathcal{T}} \big( u(t/2; z) - M_pf \cdot h_{t/2}(z)\big) h_{t/2}(d(z, x)) \right\|_{\ell^p(x)} \\
	&\leq
	C_p
	\left\|  u(t/2; \cdot) - M_pf \cdot h_{t/2}(\cdot) \right\|_{\ell^2}
	\|h_{t/2}\|_{\ell^2}.
\end{align*}
 Then, in view of \eqref{eq:heatLpnorms}, we have for all $p \in (2, \infty]$ and $t>1$,
\begin{equation*}
	\|h_{t/2}\|_{\ell^2}^2 \asymp t^{-\frac{3}{2}} e^{-b_2t}
	\asymp \|h_t\|_{\ell^p}.
\end{equation*}
Hence, by the first part of the proof for $p = 2$ (recall that $M_pf=\mathcal{H}f(0)$ for all $p\geq 2$, provided $f$ radial in $\ell^1_0(\mathcal{T})$), we conclude that
\begin{align*}
\frac{\|u(t; \cdot) - M_pf \cdot h_t(\cdot)\|_{\ell^p}}{\|h_t\|_{\ell^p}}
\lesssim\frac{\|u(t/2; \cdot) - M_2f \cdot h_{t/2}(\cdot)\|_{\ell^2}}{\|h_{t/2}\|_{\ell^2}}\to 0, 
\end{align*}
as $t\to \infty.$ The proof is now complete.
\end{proof}

\subsubsection{Data in weighted $\ell^1$ spaces.} Lastly, we study the $\ell^p$-convergence of the caloric functions with boundary datum from a class of functions which are not 
necessarily radial nor finitely supported. To this end, for each $p \in [1, \infty]$ we recall the weights $w_p$ in \eqref{eq:defwp}. Since $w_p(y) \geqslant 1$, we have $\ell^1(w_p)\subseteq \ell^1$.
Moreover, if $f \in \ell^1(w_p)$ then the mass function $M_pf$ is well-defined and bounded.  Indeed, if $p \in [1, 2)$, then by the definition of $M_pf$ in \eqref{eq:massfct1to2}, and using \eqref{eq:Busemann_bds}, we get
\begin{align*}
	|M_pf(x)|
	&\leq
	\sum_{y\in \mathcal{T}} |f(y)| 
	\fint_{\Omega(o,x)} q^{\frac{1}{p} \, h_{\omega}(y)} \, d\nu(\omega) \\
		&\leq \sum_{y\in \mathcal{T}} |f(y)|\, q^{\frac{1}{p} \,{|y|} } < \infty.
	\end{align*}
	On the other hand, if $p \in [2, \infty]$, then by the definition of the mass function in \eqref{eq:massfct>2} and \eqref{eq:phi_quot_bdd}, we have
	\begin{align*}
		|M_pf(x)|
		&\leq
		\sum_{y \in \mathcal{T}} |f(y)| \frac{\varphi(d(y, x))}{\varphi(d(o, x))} \\
		&\leq
		\sum_{y\in \mathcal{T}} |f(y)|\, q^{\frac{1}{2} \,{|y|} }.
	\end{align*}
	For initial data in $\ell^1(w_p)$ we have the following result.
	\begin{theorem}
		\label{thm:10}
		Let $p \in [1, \infty]$. Then if $f \in \ell^1(w_p)$, we have
		\[
		\lim_{t \to \infty}
		\frac{1}{\|h_t \|_{\ell^p}}
		\big\|e^{-t\mathcal{L}}f-M_pf \cdot h_t\|_{\ell^p} = 0.
		\]
	\end{theorem}
	\begin{proof}
		Let $p \in [1, 2)$ and $\varepsilon > 0$. Since $f w_p \in \ell^1(\mathcal{T})$, there is  a finitely supported function $g$ on
		$\mathcal{T}$ such that
		\begin{align*}
			\|f - g\|_{\ell^1} 
			&\leq \sum_{x \in \mathcal{T}} |f(x) - g(x)| \\
			&\leq \sum_{x \in \mathcal{T}} |f(x) - g(x)| w_p(x) 
			\leq \tfrac{1}{3} \varepsilon. 
		\end{align*}
		Moreover, by \eqref{eq:massfct1to2} and \eqref{eq:Busemann_bds}, for each $x \in \mathcal{T}$,  we have
		\begin{align*}
			\big| M_pf(x) - M_pg(x)\big| 
			&\leq
			\sum_{y \in \mathcal{T}} |f(y) - g(y)| 
			\fint_{\Omega(o, x)} q^{\frac{1}{p}\, {h_{\omega}(y)}}\, d\nu(\omega) \\
			&\leq  \sum_{y\in \mathcal{T}} |f(y)-g(y)|\, q^{\frac{1}{p}\,|y| }
			< \tfrac{1}{3}\varepsilon,
		\end{align*}
		thus
		\[
		\big\|M_pf \cdot h_t - M_pg \cdot h_t \big\|_{\ell^p}
		\leq
		\tfrac{1}{3}\varepsilon \|h_t\|_{\ell^p}.
		\]
		By the results of Section \ref{sec:finitelysupp} for finitely supported initial conditions, there is $t_1 \geqslant 1$,
		such that for all $t \geqslant t_1$,
		\[
		\frac{1}{\|h_t\|_{\ell^p}}
		\big\|e^{-t\mathcal{L}}g-M_pg \cdot h_t\big\|_{\ell^p}
		\leq \tfrac{1}{3} \varepsilon.
		\]
		Using the Minkowski inequality for integrals, we get
		\begin{align*}
			\big\|e^{-t\mathcal{L}}f - e^{-t\mathcal{L}}g \big\|_{\ell^p}
			&=
			\bigg\|\sum_{y \in \mathcal{T}} (f(y)-g(y)) h_t(d(y,\cdot)) \bigg\|_{\ell^p} \\
			&\leq
			\|f -g\|_{\ell^1} \|h_t\|_{\ell^p} 
			\leq \tfrac{1}{3} \varepsilon  \|h_t\|_{\ell^p}.
		\end{align*}
		Hence, for $t \geqslant t_1$,
		\begin{align*}
			\big\|e^{-t\mathcal{L}}f-M_pf \cdot h_t\|_{\ell^p}
			&\leq
			\big\|e^{-t\mathcal{L}}f - e^{-t\mathcal{L}}g \big\|_{\ell^p}
			+
			\big\|e^{-t\mathcal{L}}g-M_pg \cdot h_t \big\|_{\ell^p}
			\\
			&+\big\|M_pf \cdot h_t - M_pg \cdot h_t \big\|_{\ell^p}
			\\
			&\leq \varepsilon \|h_t\|_{\ell^p}.
		\end{align*}
	
		The proof for $p \in [2, \infty]$ follows by the same line of reasoning provided one uses \eqref{eq:phi_quot_bdd}. We omit the details.
	\end{proof}

\end{document}